\documentclass{cmip-l}

\usepackage{amsmath,amssymb,amsfonts,amscd,graphicx}

\newtheorem{theorem}{Theorem}[section]

\theoremstyle{definition}

\theoremstyle{remark}
\newtheorem{remark}[theorem]{Remark}

\numberwithin{equation}{section}

\begin{document}
\theoremstyle{plain}
\newtheorem{thm}{Theorem}
\newtheorem{lem}[thm]{Lemma}
\newtheorem{cor}[thm]{Corollary}
\newtheorem{prop}[thm]{Proposition}
\newtheorem{defn}[thm]{Definition}
\newtheorem{ex}[thm]{Example}
\newtheorem{conj}[thm]{Conjecture}
\numberwithin{equation}{section}
\numberwithin{thm}{section}
\newcommand{\eq}[2]{\begin{equation}\label{#1}#2 \end{equation}}
\newcommand{\ml}[2]{\begin{multline}\label{#1}#2 \end{multline}}
\newcommand{\ga}[2]{\begin{gather}\label{#1}#2 \end{gather}}
\newcommand{\mc}{\mathcal}
\newcommand{\mb}{\mathbb}
\newcommand{\surj}{\twoheadrightarrow}
\newcommand{\inj}{\hookrightarrow}
\newcommand{\red}{{\rm red}}
\newcommand{\codim}{{\rm codim}}
\newcommand{\rank}{{\rm rank}}
\newcommand{\Pic}{{\rm Pic}}
\newcommand{\Div}{{\rm Div}}
\newcommand{\Hom}{{\rm Hom}}
\newcommand{\im}{{\rm im}}
\newcommand{\Spec}{{\rm Spec \,}}
\newcommand{\Sing}{{\rm Sing}}
\newcommand{\Char}{{\rm char}}
\newcommand{\Tr}{{\rm Tr}}
\newcommand{\Gal}{{\rm Gal}}
\newcommand{\Min}{{\rm Min \ }}
\newcommand{\Max}{{\rm Max \ }}
\newcommand{\ti}{\times }
% Skriptbuchstaben
\newcommand{\sA}{{\mathcal A}}
\newcommand{\sB}{{\mathcal B}}
\newcommand{\sC}{{\mathcal C}}
\newcommand{\sD}{{\mathcal D}}
\newcommand{\sE}{{\mathcal E}}
\newcommand{\sF}{{\mathcal F}}
\newcommand{\sG}{{\mathcal G}}
\newcommand{\sH}{{\mathcal H}}
\newcommand{\sI}{{\mathcal I}}
\newcommand{\sJ}{{\mathcal J}}
\newcommand{\sK}{{\mathcal K}}
\newcommand{\sL}{{\mathcal L}}
\newcommand{\sM}{{\mathcal M}}
\newcommand{\sN}{{\mathcal N}}
\newcommand{\sO}{{\mathcal O}}
\newcommand{\sP}{{\mathcal P}}
\newcommand{\sQ}{{\mathcal Q}}
\newcommand{\sR}{{\mathcal R}}
\newcommand{\sS}{{\mathcal S}}
\newcommand{\sT}{{\mathcal T}}
\newcommand{\sU}{{\mathcal U}}
\newcommand{\sV}{{\mathcal V}}
\newcommand{\sW}{{\mathcal W}}
\newcommand{\sX}{{\mathcal X}}
\newcommand{\sY}{{\mathcal Y}}
\newcommand{\sZ}{{\mathcal Z}}
% Sonderbuchstaben mit Doppellinie
\newcommand{\A}{{\Bbb A}}
\newcommand{\B}{{\Bbb B}}
\newcommand{\C}{{\Bbb C}}
\newcommand{\D}{{\Bbb D}}
\newcommand{\E}{{\Bbb E}}
\newcommand{\F}{{\Bbb F}}
\newcommand{\G}{{\Bbb G}}
\renewcommand{\H}{{\Bbb H}}
\newcommand{\I}{{\Bbb I}}
\newcommand{\J}{{\Bbb J}}
\newcommand{\M}{{\Bbb M}}
\newcommand{\N}{{\Bbb N}}
\renewcommand{\P}{{\Bbb P}}
\newcommand{\Q}{{\Bbb Q}}
\newcommand{\R}{{\Bbb R}}
\newcommand{\T}{{\Bbb T}}
\newcommand{\U}{{\Bbb U}}
\newcommand{\V}{{\Bbb V}}
\newcommand{\W}{{\Bbb W}}
\newcommand{\X}{{\Bbb X}}
\newcommand{\Y}{{\Bbb Y}}
\newcommand{\Z}{{\Bbb Z}}
\newcommand{\pic}{{\text{Pic}(C,\sD)[E,\nabla]}}
\newcommand{\ocd}{{\Omega^1_C\{\sD\}}}
\newcommand{\oc}{{\Omega^1_C}}
\newcommand{\al}{{\alpha}}
\newcommand{\be}{{\beta}}
\newcommand{\ta}{{\theta}}
\newcommand{\ve}{{\varepsilon}}
\newcommand{\lr}[2]{\langle #1,#2 \rangle}
\newcommand{\nnn}{\newline\newline\noindent}
\newcommand{\nn}{\newline\noindent}
\newcommand{\snote}[1]{{\bf #1}}
\title{Feynman Amplitudes in Mathematics and Physics}


%    author one information
\author{Spencer Bloch}
\address{5765 S Blackstone Ave., Chicago, IL, 60637}
\email{spencer\_bloch@yahoo.com}

\subjclass[2000]{Primary }
%    The 2010 edition of the Mathematics Subject Classification is
%    now available.  If you are citing a classification from the
%    new scheme, use the following input coding instead.
%\subjclass[2010]{Primary }

\date{August 23, 2015}

\begin{abstract}These are notes of lectures given at the CMI conference in August, 2014 at ICMAT in Madrid. The focus is on some mathematical questions associated to Feynman amplitudes, including Hodge structures, relations with string theory, and monodromy (Cutkosky rules). 
\end{abstract}

\maketitle

\tableofcontents
\newpage
\section{Introduction}
What follows are notes from some lectures I gave at a Clay Math conference at CMAT in Madrid in August, 2014. I am endebted to the organizers for inviting me, and most particularly to Jos\'e Burgos Gil for giving me his notes. 

The subject matter here is (a slice of) modern physics as viewed by a mathematician. To understand what this means, I told the audience about my grandson who, when he was 4 years old, was very much into trains. (He is 5 now, and more into dinosaurs.) For Christmas, I bought him an elaborate train set. At first, the little boy was at a complete loss. The train set was much too complicated. Amazingly, everything worked out! Even though the train itself was hopelessly technical and difficult, the box the train came in was fantastic; with wonderful, imaginative pictures of engines and freight cars.   Christmas passed happily in fantasy play inspired by the images of the trains on the box.  Modern physics is much too complicated for anyone but a ``trained'' and dedicated physicist to follow. However, ``the box it comes in'', the superstructure of mathematical metaphor and analogy which surrounds it, can be a delightful inspiration for mathematical fantasy play. 

The particular focus of these notes is the Feynman amplitude. Essentially, to a physical theory the physicist associates a lagrangian, and to the lagrangian a collection of graphs, and to each graph a function called the amplitude on a space of external momenta. Section \ref{configsect} develops the basic algebra of the first and second Symanzik polynomials in the context of configurations, i.e. invariants associated to a finite dimensional based vector space $\Q^E$ with a given subspace $H\subset \Q^E$. Section \ref{graph1} considers the first Symanzik (aka Kirchoff polynomial) $\psi_G$ in the case $E=\text{edges}(G)$ and $H=H_1(G)$ for a graph $G$. $\psi_G= \det(M_G)$ is a determinant, so the {\it graph hypersurface} $X_G:\psi_G=0$ admits a birational cover $\pi: \Lambda_G \to X_G$ with fibre over $x\in X_G$ the projective space associated to $\ker M_G(x)$.  The main result in this section is theorem \ref{patterson}. A classical theorem of Riemann considers singularities of the theta divisor $\Theta\subset J^{g-1}(C)$ where $C$ is a Riemann surface of genus $g$. Here $J^{g-1}(C)$ is the space of rational equivalence classes of divisors of degree $g-1$ on $C$ and $\Theta \subset J^{g-1}$ is the subspace of effective divisors. One shows that $\Theta$ is (locally) defined by the vanishing of a determinant, and the evident map $\text{Sym}^{g-1}(C) \to \Theta$ has fibres the projectivized kernels. Riemann's result is that the dimension of the fibre over $x\in \Theta$ is one less than the multiplicity of $x$ on $\Theta$. Theorem \ref{patterson} is the analogous result for $X_G$, viz. $\dim(\pi^{-1}(x)) = \text{Mult}_x(X_G)-1$. 

Section \ref{lambda} considers in more detail the Hodge structure of $\Lambda_G$. It turns out that whereas the Hodge structure for $X_G$ is subtle and complicated, the Hodge structure on $\Lambda_G$ is fairly simple. In particular, $\Lambda_G$ is mixed Tate. 

Sections \ref{2symanzik}-\ref{limit} focus on the amplitude $A_G$, viewed as a multi-valued function of external momenta. Various standard formulas for $A_G$ are given in section \ref{2symanzik}.  Sections \ref{riemannsurfaces}-\ref{limit} develop from a Hodge-theoretic perspective a result relating $A_G$ to a limit in string theory when the string tension $\alpha' \to 0$. Associated to $G$ is a family of stable rational curves which are unions of Riemann spheres with dual graph $G$.  We view this family as lying on the boundary of a moduli space of pointed curves of genus $g=h_1(G)$. We show how external momenta are associated to limits at the boundary of marked points, and we explicit the second Symanzik as a limit of heights. The amplitude $A_G$ then becomes an integral over the space of nilpotent orbits associated to the degeneration of the Hodge structure given by $H^1_{\text{Betti}}$ of the curves. This is joint work with Jos\'e Burgos Gil, Omid Amini, and Javier Fresan.  We are greatly endebted to P. Vanhove and P. Tourkine for their insights, (\cite{T} and references cited there.) I met my co-authors at the conference, and this work grew out of an attempt to make sense of vague and imprecise suggestions I had made in the lectures. 

The last topic (sections \ref{cutrules} - \ref{cutthm}) concerns joint work with Dirk Kreimer on Cutkosky Rules. The amplitude $A_G$ is a multi-valued function of external momenta. Cutkosky rules give a formula for the {\it variation} as external momenta winds around the {\it threshold divisor} which is the discriminant locus in the space of external momenta for the quadric propagators. Over the years there has been some uncertainty in the physics community as to the precise conditions for Cutkosky rules to apply. Using a modified notion of vanishing cycle due to F. Pham, we prove that Cutkosky rules do apply in the case of what are called {\it physical singularities}. 

Perhaps a word about what is {\it not} in these notes. In recent years, there has been enormous progress in calculating Feynman amplitudes. The work I should have liked to talk about is due to Francis Brown and Oliver Schnetz \cite{BS} which exhibits an algorithm for calculating many Feynman amplitudes and gives a beautiful example where the algorithm fails. In this case the amplitude is the period of a modular form. The whole subject of polylogarithms, which is intimately linked to amplitudes, is not addressed at all. I apologize; there simply wasn't time. 

Finally, a shout-out to the organizers, most particularly Kurusch Ebrahimi-Fard who worked very very hard on both the practical and scientific aspects of the conference. 

\section{Configuration Polynomials}\label{configsect}
Let $H$ be a finite dimensional vector space of dimension $g$ over a field $k$, and suppose we are given a finite set $E$ and an embedding $\iota: H \inj k^E$, where $k^E:= \{\sum_{e\in E} \kappa_ee\ |\ \kappa_e \in k\}$. Write $W=k^E/H$. Let $e^\vee: k^E \to k$ be the evident functional that simply takes the $e$-th coordinate of a vector, and write $e^\vee$ as well for the composition $H \inj k^E \to k$. The function $h\mapsto (e^\vee(h))^2$ defines a rank $1$ quadratic form $e^{\vee,2}$ on $H$. If we fix a basis of $H$, we can identify the quadratic form with a rank $1$ $g\times g$ symmetric matrix $M_e$ in such a way that, thinking of elements of $H$ as column vectors, we have
\eq{1}{e^{\vee,2}(h) = h^tM_eh.
}

Let $A_e,\ e\in E$ be variables, and consider $M:=\sum_{e\in E}A_eM_e$. We can interpret $M$ as a $g\times g$ symmetric matrix with entries that are linear forms in the $A_e$. More canonically, $M: H \to H^\vee$ is independent of the choice of basis of $H$. 
\begin{defn}The first Symanzik polynomial $\psi(H,\{A_e\})$ associated to $H\inj k^E$ is the determinant
\eq{2}{\psi(H,\{A_e\}) := \det(M).
}
\end{defn}
\begin{remark}A different choice of basis multiplies $\psi(H,\{A_e\})$ by an element in $k^{\times,2}$. Indeed, $M$ is replaced by $N^tMN$ where $N$ is a $g\times g$ invertible matrix with entries in $k$. 
\end{remark}

For $w\in W$, define $H_w$ to be the inverse image in $k^E$ of $w\in W=k^E/H$. We have $H\subset H_w \subset k^E$, and we can calculate the first Symanzik polynomial of $\psi(H_w,\{A_e\})$. 
\begin{defn}\label{second_symanzik} The second Symanzik polynomial
\eq{}{\phi(H,w,\{A_e\}):= \psi(H_w, \{A_e\}).
}
\end{defn}
\begin{lem}(i) The first Symanzik $\psi(H,\{A_e\})$ is homogeneous of degree $g=\dim H$ in the $A_e$. \nn
(ii) The second Symanzik $\phi(H,w,\{A_e\})$ is homogeneous of degree $g+1$ in the $A_e$ and is quadratic in $w$. 
\end{lem}
\begin{proof} The assertions about homogeneity in $A_e$ are clear. To check that $\phi$ is quadratic in $w$, it suffices to note that 
\eq{4}{\phi(H,w,\{A_e\}) = \det\Big(\sum A_e\begin{pmatrix}M_e & W_e \\
{}^tW_e & Q(W_e)\end{pmatrix}\Big)
} 
is quadratic in $W_e$, where $W_e=(w_{e,1},\dotsc,w_{e,g})$ is a column vector and $Q(W_e)=\sum_1^n w_{e,i}^2$. This is straightforward, expanding the determinant by the last row for example. 
\end{proof}
\begin{remark}In fact, one can do a bit better. Let $V$ be a vector space and let $q$ be a quadratic form on $V$. We can take the $W_e=(v_{e,1},\dotsc,v_{e,g})$ to have entries in $V$ and define $Q(W_e):= \sum_1^gq(v_{e,i})$. Using the quadratic form $q$, one can make sense of the determinant expression on the right in \eqref{4} and define $\phi(H,w,\{A_e\})$ for $w\in W\otimes V$. Typically, in physics $V=\R^D$ is space-time and $q$ is the Minkowski metric. 
\end{remark}
\begin{remark}\label{det_formula} To relate the second Symanzik to the height (see sections \ref{2symanzik}-\ref{limit}), it will be convenient to rewrite the above determinant expression as a bilinear form. Suppose we are given a symmetric $(g+1)\times (g+1)$ matrix of the form 
\eq{}{A:= \begin{pmatrix} M & W\\
W^t &  S
\end{pmatrix}
}
where $M$ is $g\times g$ and invertible, $W$ is a column vector of length $g$, and $S$ is a scalar. Recall the classical formula $(\det M)M^{-1} = \text{adj}(M)$ where $\text{adj}(M)$ is the matrix of minors.  Then 
\eq{}{\det A = -W^{t}\text{adj}(M)W+S\det M.
}
In the context of the previous remark, if $W$ has entries in a vector space $V$ with a quadratic form $Q$, the determinant in \eqref{4} can be written
\eq{}{\phi(H,w,\{A_e\})/\psi(H,\{A_e\}) = -W^{t}M^{-1}W+Q(W^t\cdot W)
}
where the product $W^{t}M^{-1}W$ is interpreted via the quadratic form. 
\end{remark}
\section{The First Symanzik Polynomial for Graphs}\label{graph1}
Of primary importance for us will be configurations associated to graphs. 
\begin{defn} A graph $G$ is determined by sets $E=E(G)$ (edges) and $V=V(G)$ (vertices) together with a set $\Lambda$ which we can think of as the set of {\rm half edges} of $G$. We are given a diagram of projections
\eq{}{E \xleftarrow{p}\Lambda \xrightarrow{q} V.
}
We assume $p$ is surjective and $p^{-1}(e)$ has $2$ elements for all $e\in E$.  
\end{defn} 
(This definition may seem a bit fussy, but half edges are very useful when one wants to talk about automorphisms of a graph. For example, there is a unique graph $G$ with one edge and one vertex, and $\text{Aut}(G)=\Z/2\Z$.) An {\it orientation} of $G$ is an ordering of $p^{-1}(e)$ for all edges $e$, i.e. a map $\Lambda \to E\times \{-1,1\}$ satisfying obvious conditions. For an oriented graph $G$ there is a boundary map $\partial: \Z^E \to \Z^V,\ e \mapsto \partial^+(e)-\partial^-(e)$. We define the homology of $G$ via the exact sequence
\eq{6}{0 \to H_1(G,\Z) \to \Z^E \xrightarrow{\partial} \Z^V \to H_0(G,\Z) \to 0. 
}
The homology of $G$ with coefficients in an abelian group $A$ is defined similarly. 
\begin{defn}The graph polynomial $\psi_G(\{A_e\}_{e\in E(G)})$ (also sometimes called the Kirchoff polynomial or the first Symanzik polynomial of $G$) for $G$ is the first Symanzik polynomial for the configuration $H_1(G,\Z)\subset \Z^E$. 
\end{defn}

Let $G$ be a graph and let $e\in G$ be an edge. The graph $G-e$ (resp. $G/e$) is obtained by cutting (resp. contracting) the edge $e$ in $G$. The general notion of cutting or contracting in a configuration is explained by the diagram 
\eq{7}{\begin{CD}0 @>>> H(\text{cut }e) @>>> k^{E-\{e\}} \\
@. @VV\text{inject}V @VV\text{inject} V \\
0 @>>> H @>>> k^E \\
@. @VVV @VV\text{surject}V \\
0 @>>> H(\text{shrink }e) @>>> k^{E-e}
\end{CD}
}
\begin{prop}\label{prop1} (i) $\psi(H(\text{cut }e)) = \frac{\partial}{\partial A_e}\psi(H)$. \nn
(ii) $\psi(H(\text{shrink }e) = \psi(H)|_{A_e=0}$ unless $e$ is a tadpole (i.e. the edge $e$ has only one vertex). If $e$ is a tadpole, $\psi(H)|_{A_e=0} = 0$. \nn
(iii) $\psi(H)$ has degree $\le 1$ in every $A_e$. 
\end{prop}
\begin{proof}(i) and (ii) are straightforward from \eqref{7}. For (iii), note $\psi$ depends only upto scale on the choice of basis of $H$. We choose a basis $h_1,\dotsc,h_g$ such that $e^\vee(h_i)=0$ for $2\le i\le g$. The matrix $M_e$ is then $g\times g$ diagonal with $g-1$ zeroes and a single $1$ in position $(1,1)$. The matrix $\sum_{\ve \in E} A_\ve M_\ve$ then involves $A_e$ only in position $(1,1)$, and (iii) follows. 
\end{proof}
In the case of the graph polynomial of a graph $G$, one can be more precise. A {\it spanning tree} $T\subset G$ is a connected subgraph of $G$ with $V(T) = V(G)$ and $H_1(T,\Z)=(0)$. 
\begin{prop}Let $G$ be a graph. Then the graph polynomial can be written
\eq{}{\psi_G = \sum_{T\subset G}\prod_{e\not\in T}A_e.
}
(Here $T$ runs through all spanning trees of $G$. 
\end{prop}
\begin{proof}We know by proposition \ref{prop1} (iii) that every monomial in $\psi_G$ is a product of $g=h_1(G)$ distinct edge variables $A_e$. For a given set $\{e_1,\dotsc,e_g\}$ of distinct edges, we know by proposition \ref{prop1} (ii) that this term is exactly the graph polynomial of $G$ with all the edges $\ve \not\in \{e_1,\dotsc,e_g\}$ shrunk. If the edges not in $\{e_1,\dotsc,e_g\}$ form a spanning tree for the graph, shrinking them will yield a {\it rose} with $g$ loops, i.e. a union of $g$ tadpoles. The graph polynomial of such a graph is simply $\prod_1^g A_{e_i}$ so the coefficient of this product in $\psi_G$ is $1$. On the other hand, if the edges $A_\ve$ not in $\{e_1,\dotsc,e_g\}$ do not form a spanning tree, they must necessarily contain a loop and so setting them to $0$ kills $\psi_G$ and there is no term in $A_{e_1},\dotsc,A_{e_g}$. 
\end{proof}

We continue to assume $G$ is a graph. We write $n=\# E(G)$ and $g=\dim H_1(G)$. The hypersurface $X_G:\psi_G=0$ in $\P^{n-1}$ is the {\it graph hypersurface}. The symmetric matrix $M:= \sum_e A_eM_e$ defines a linear map $\bigoplus_g \sO_{\P^{n-1}} \to \bigoplus_g \sO_{\P^{n-1}}(1)$. Define
\eq{}{\Lambda_G := \{(a,\beta)\in \P^{n-1}\times \P^{g-1}\ |\ M_a(\beta) = 0\}
}
Note $\Lambda_G \subset X_G\times \P^{g-1} \subset \P^{n-1}\times \P^{g-1}$. 
\begin{prop}\label{Lambdaprop}(i) There exist coherent sheaves $\sE$ on $X_G$ and $\sF$ on $\P^{g-1}$ such that $\Lambda_G \cong \text{Proj(Sym}(\sE)) \cong \text{Proj(Sym}(\sF))$. \nn
(ii) $\Lambda_G$ is a reduced, irreducible variety of dimension $n-2$ which is a complete intersection of codimension $g$ in $\P^{n-1}\times \P^{g-1}$. The projection $\pi:\Lambda_G \to X_G$ is birational. 
\end{prop}
\begin{proof}Define $\sE$ by the presentation
\eq{}{H_1(G)\otimes \sO_{X_G} \xrightarrow{M} H^1(G)\otimes \sO_{X_G}(1) \to \sE \to 0. 
}
Here for $a\in X_G$, $M_a=\sum a_eM_e$. 

For $\sF$, the map which is given over $\beta\in \P^{g-1}$ by $a\mapsto \sum_E a_e(e^\vee(\beta))e^\vee$ dualizes to a presentation
\eq{}{\sO_{\P^{g-1}}^g \to \sO_{\P^{g-1}}^n(1) \to \sF \to 0. 
}
The fibre $\sF_\beta$ is the quotient of $k^{n,\vee}$ by the space of functionals of the form $a\mapsto \sum_e a_ee^\vee(\beta)e^\vee$. We have $\dim \sF_\beta=n-g+\ve(\beta)$ where 
$\ve(\beta)$ is the codimension in $k^{g,\vee}$ of the span of $\{e^\vee |\ e^\vee(\beta)\neq 0\}$. Since the $e^\vee$ span $H_1(G)^\vee$ it follows that for $\beta$ general we have $\ve(\beta)=0$ so $\Lambda_G=\text{Proj(Sym)}(\sF)$ has dimension n-2. Since the fibre of $\sE$ over a point $a$ is the kernel of $M_a$  it is non-zero for $a\in X_G$, whence $\Lambda_G \surj X_G$ with fibres projective spaces. Since the two varieties have the same dimension, it follows that $\Lambda_G \to X_G$ is birational. 

Finally, to realize $\Lambda_G$ as a complete intersection, let $A_e$ (resp. $B_j$) be a basis for the homogeneous coordinates on $\P^{n-1}$ (resp. $\P^{g-1}$). We can think of $B_j \in H_1(G)$. Write $w_{e,j}=e^\vee(B_j)$. Then the defining equations are
\eq{15}{0 = \sum_{j=1}^g \sum_E A_ew_{e,j}B_jw_{e,i}=0;\quad 1\le i\le g. 
}
\end{proof}

We will further investigate the motives of $X_G$ and $\Lambda_G$ in the sequel. There is an interesting analogy between $X_G$ and $\Theta$, the theta divisor of a genus $g$ Riemann surface $C$. Both are determinental varieties, and $\Lambda_G\to X_G$ corresponds to $\text{Sym}^{g-1}C \xrightarrow{\pi} \Theta$. One has that $\Theta$ is a divisor on the jacobian $J(C)$, and a classical theorem of Riemann states that the multiplicity of $\Theta$ at a point $x$ is equal to $1+\dim \pi^{-1}(x)$. We will show below that the same result holds for graph hypersurfaces $X_G$. 
\begin{ex} Consider a determinental variety defined by the determinant of a diagonal variety 
$$X:\det\begin{pmatrix} f_1&0&\hdots & 0\\ 0 & f_2 & \hdots & 0 \\ \vdots & \vdots & \hdots & \vdots \\ 0 & 0 & \hdots & f_g\end{pmatrix}=0$$
Then $X$ satisfies Riemann's theorem iff the zero sets of the $f_i$ are smooth and transverse. 
\end{ex}
\begin{thm}[E. Patterson \cite{Pat}]\label{patterson} Let $G$ be a graph, and let $\pi: \Lambda_G \to X_G$ be the birational cover of $X_G$ defined above. For $x\in X_G$ the fibre $\pi^{-1}(x)$ has dimension equal to $\text{mult}_x(X_G)-1$. 
\end{thm}
\begin{proof}Define $X_p:=\{x\in X_G\ |\ \dim(\pi^{-1}(x)\ge p\}$ and $X(p):=\{x\in X_G\ |\ \text{mult}_x(X_G)\ge p+1$. We want to show $X_p=X(p)$. 
\begin{lem}$X_p\subset X(p)$. 
\end{lem}
\begin{proof}[proof of lemma] We have 
$$x\in X_p \Leftrightarrow \dim(\ker(M_x=\sum x_eM_e: H_1(G) \to H_1(G)))\ge p+1.
$$
On the other hand, $X(p)$ is defined by the vanishing of all $p$-fold derivatives $\frac{\partial^p}{\partial A_{e_1}\cdots \partial A_{e_p}}\psi_G=\psi_{G-\{e_1,\dotsc,e_p\}}$. If we associate to $G$ the quadratic form $\sum x_ee^{\vee,2}$ on $H_1(G)$, then $X_p$ is the set of $x$ for which the null space of this form has dimension $\ge p+1$. The quadratic form associated to $G-\{e_1,\dotsc,e_p\}$ is simply $\sum x_ee^{\vee,2}|H_1(G)\cap\bigcap_{i=1}^p\{e_i^\vee=0\}$. This restricted form cannot be nondegenerate if the null space of the form on $H_1(G)$ had dimension $\ge p+1$. 
\end{proof}
\begin{lem}Let $Q$ be a quadratic form on a vector space $H$. Let $N\subset H$ be the null space of $Q$. We assume $\dim N=s>0$.  Suppose we are given an embedding $H\inj k^E$ for a finite set $E$. For $e\in E$ write $e^\vee: H \to k$ for the corresponding functional. Then there exists a subset $\{e_1,\dotsc,e_s\}\subset E$ such that $Q|\{e_1^\vee=\cdots=e_p^\vee=0\}$ is non-degenerate. 
\end{lem}
\begin{proof}[proof of lemma] It suffices to take $e_1,\dotsc,e_s$ such that $N\cap\{e_1^\vee=\cdots=e_s^\vee=0\} = (0)$. Indeed, if $L$ is any codimension $s$ subspace with $L\cap N=(0)$ we will necessarily have $Q|L$ nondegenerate. Since $H=L\oplus N$, any $\ell$ in the null space of $Q|L$ will necessarily be orthogonal to $L\oplus N = H$. 
\end{proof}
We return to the proof of the theorem. By the first lemma wie have $X_i\subset X(i)$. As a consequence of the last lemma we see that 
\eq{}{X_i-X_{i+1} \subset X(i)-X(i+1).
}
In other words, if the nullspace has dimension exactly $i+1$, then there exists some $(i+1)$-st order partial which doesn't vanish. Taking the disjoint union we get $X_0-X_j \subset X(0)-X(j)$ for any $j$. Since $X_0=X(0)$ it follows that $X(j)\subset X_j$. This completes the proof. 
\end{proof}
\section{$X_G$ and $\Lambda_G$}\label{lambda}

It turns out that the hypersurface $X_G:\psi_G=0$  in $\P^{n-1}$ is quite subtle and complicated, while the variety $\Lambda_G$ introduced above is rather simple. Recall we have a birational map $\pi: \Lambda_G \to X_G$ and the fibre $\pi^{-1}(a)$ is the projectivized kernel of $M_a=\sum a_eM_e$. 

Our arguments at this point are completely geometric but for simplicity we focus on the the case of Betti cohomology and varieties over the complex numbers. The key point is
\begin{prop}\label{mixedtate} The Hodge structure on Betti cohomology $H^*(\Lambda_G,\Q)$ is mixed Tate. 
\end{prop}
Recall Betti cohomology of a variety over $\C$ carries a Hodge structure. 
\begin{defn}A Hodge structure on a finite dimensional $\Q$-vector space $H$ is a pair of filtrations $(W_*, F^*)$ with $W_*H_\Q$ a finite increasing filtration (separated and exhaustive) and $F^* = F^*H_\C$ a finite (separated and exhaustive) decreasing filtration. The filtration induced by $F$ on $gr_p^WH_\C$ should be $p$-opposite to its complex conjugate, meaning that
\eq{}{gr_p^WH_\C = \bigoplus F^qgr_p^WH_\C\cap \overline F^{p-q}gr_p^WH_\C 
}
\end{defn}
A Hodge structure is {\it pure} if its weight filtration has a single non-trivial weight. 
\begin{ex}The Tate Hodge structures $\Q(n)$ are one dimensional $\Q$-vector spaces with weight $\Q(n)=W_{-2n}\Q(n)$ and Hodge filtration $F^{-n}\Q(n)_\C = \Q(n)_\C\supset F^{-n+1}=(0)$. 
\end{ex}
\begin{defn}A Hodge structure $H$ is mixed Tate if $gr^W_{-2n}H= \bigoplus \Q(n)$ for all $n$. 
\end{defn}
\begin{lem}Let $V$ be a variety over $\C$. Assume $V$ admits a finite stratification $V=\amalg V_i$ by Zariski locally closed sets such that $H_c^*(V_i)$ is mixed Tate for all $i$. (Here $H_c$ is cohomology with compact supports.) Then $H^*_c(V)$ is mixed Tate.
\end{lem}
\begin{proof}The functor $H^* \mapsto gr^W_pH^*$ is exact on the category of Hodge structures. We apply this functor to the spectral sequence which relates $H_c^*(V_i)$ to $H_c^*(V)$ and deduce a spectral sequence converging to $gr^W_pH^*(V)$ with initial terms direct sums of $\Q(p)$. Since extensions of $\Q(p)$ are all split, it follows that $gr^W_pH^*(V)$ is a direct sum of $\Q(p)$ so $H^*_c(V)$ is mixed Tate. 
\end{proof} 
Note of course that $H^*_c(V) = H^*(V)$ if the variety $V$ is proper. 
\begin{proof}[Proof of Proposition \ref{mixedtate}] Let $\ve: \P^{g-1} \to \N$ be as in the proof of Proposition \ref{Lambdaprop}. Define
\eq{}{T^m=\{\beta\ |\ \ve(\beta)\ge m\}.
}
It is clear that $T^m \subset \P^{g-1}$ is closed, and $T^{m+1}\subset T^m$. The sets $S^m:=T^m-T^{m+1}$ form a locally closed stratification on $\P^{g-1}$. Let $\sF$ be the constructible sheaf on $\P^{g-1}$ defined in Proposition \ref{Lambdaprop}. The fibres of $\sF$ over $S^m$ have constant rank, so $\sF|_{S^m}$ is a vector bundle and $\Lambda|_{S^m}$ is a projective bundle. It will suffice by the lemma to show $H^*_c(\Lambda|_{S^m})$ is mixed Tate, and by the projective bundle theorem this will follow if we show $H^*_c(S^m)$ is mixed Tate. 

The set $T^m$ can be described as follows. Let $Z\subset 2^E$ be the set of all subsets $z\subset E$ such that the span of $e^\vee|H_1(G), e\in z$ has codimension $<m$ in $H^1(G)$. Then $T^m$ is the set of $\beta$ such that $e^\vee(\beta)=0$ for at least one $e^\vee \in z$. Said another way, for any subset $W\subset E$ containing at least one edge from each $z\in Z$, let $L_W\subset \P^{g-1}$ be the set of those $\beta$ such that $e^\vee(\beta)=0$ for all $e\in W$. Then $T^m=\bigcup L_W$ is the union of the $L_W$. since the cohomology of a union of linear spaces is mixed Tate, we see that $H^*(T^m)$ is mixed Tate. Finally, from the long exact sequence relating the cohomologies of $T^m, T^{m+1}$ to the compactly supported cohomology of $S^m$ we deduce that $H^*_c(S^m)$ is mixed Tate as well. 
\end{proof}

We had mentioned the analogy between $X_G$ and the theta divisor $\Theta \subset J(C)$ of an algebraic curve $C$ of genus $n-1$. From this point of view, the birational map $\pi:\Lambda_G\to X_G$ is analogous to the map $\text{Sym}^{n-2}C \to \Theta$. Unlike the curve case, however, $\Lambda_G$ is not usually smooth. To understand this, we consider partitions $E(G) = E'\amalg E''$. Let $G', G'' \subset G$ be the unions of the corresponding edge sets. We say our partition is {\it non-trivial on loops} if neither $\{e^\vee\}_{e\in E'}$ nor $\{e^\vee\}_{e\in E''}$ span $H^1(G)$. 

We have seen that $\Lambda_G = \text{Proj}(\sF)$ is a projective fibre space over $\P^{g-1}$. The general fibres have dimension $n-g-1$. Of course, over the open set with fibres of dimension exactly $n-g-1$, $\Lambda_G$ is a projective bundle, hence smooth. Singularities can occur only when the fibre dimension jumps. 
\begin{prop}\label{singsprop}(i) The fibre of $\Lambda_G$ over $\beta \in \P^{g-1}$ has dimension $> n-g-1$ iff there exists a partition $E=E'\amalg E''$ which is non-trivial on loops such that $e^\vee(\beta)=0$ for all $e\in E'$. \nn
(ii) $\Lambda_G$ is singular iff there exists a partition $E=E'\amalg E''$ which is non-trivial on loops. 
\end{prop}
\begin{proof} For (i), we may take $E''=\{e\ |\ e^\vee(\beta)\neq 0\}$. Assertion (i) is now straightforward from the definition of $\ve(\beta)$ in the proof of Proposition \ref{Lambdaprop}.

For (ii), note that if no such partition exists, then $\Lambda_G$ is a projective bundle over $\P^{g-1}$, hence smooth, so the existence of a partition is certainly necessary. Assume such a partition exists. The equation for $\Lambda_G$ can be written in vector form $\sum_{e\in E} a_ee^\vee(\beta)e^\vee|_{H_1(G)}$ (compare \eqref{15}). We take $\beta$ so $e^\vee(\beta)=0$ for all $e\in E'$ and we take $a_e=0$ for all $e\in E'$. (The value of the equation for such $\beta$ does not depend on the choice of $a_e,\ e\in E'$, so this is a free choice.) It is clear that the partial derivatives at such a point with respect to the $a_e$ and also with respect to coordinates on $H_1(G)$ all vanish, so this is a singular point. 
\end{proof}
\begin{remark}The singular structure of $\Lambda_G$ is more complicated than the above argument suggests because it may happen that there exists $\emptyset \neq F \subsetneq E'$ such that $e^\vee, e\in F\amalg E''$ still does not span $H^1(G)$. In such a case, it suffices to take $a_e=0$ for $e\in E'-F$. 
\end{remark}
\begin{ex}[Wheel and spoke graphs with $3$ and $4$ edges.]  (i) The wheel with $3$ spokes graph $G_3$ has vertices $1, 2, 3, 4$ and edges 
\eq{}{\{1,2\},\{2,3\},\{3,1\},\{1,4\},\{2,4\},\{3,4\}.
}
It has $3$ loops, but it is easy to check there are no partitions of the edges which are non-trivial on the loops. It follows from Prop. \ref{singsprop} that $\Lambda_{G_3}$ is a $\P^2$-bundle over $\P^2$ and hence non-singular. \nn
(ii) The wheel with $4$ spokes $G_4$ has $5$ vertices, $4$ loops, and $8$ edges:
\eq{}{\{1,2\},\{2,3\},\{3,1\},\{4,1\},\{1,5\},\{2,5\}, \{3,5\},\{4,5\}.
}
With the aid of a computer, one can show that $\Lambda_{G_4} \to \P^3$ is a $\P^3$-bundle over $\P^3-4$ points. Over the $4$ points, the fibre jumps to $\P^4$. 
\end{ex}

We next calculate $gr^WH^*(\Lambda_G)$ for a general graph $G$. Let $f: \Lambda_G \to \P^{n-1}$ be the projection. 
\begin{lem}The sheaves $R^af_*\Q_\Lambda$ are zero for $a$ odd. For $a=2b$, we have
\eq{}{R^af_*\Q_\Lambda \cong \Q(-b)|_{S_b}
}
where $S_b \subset \P^{g-1}$ is the closed set where the fibre dimension of $f$ is $\ge b$. In particular, $S_b=\P^{g-1}$ for $b\le n-g-1$. 
\end{lem}
\begin{proof}Let $p: \Lambda_G \to \P^{n-1}$ be the other projection (recall $\Lambda_G \subset \P^{g-1}\times \P^{n-1}$.) We have a map of sheaves on $\Lambda_G$
\eq{}{p^*: H^a(\P^{n-1},\Q)_{\P^{g-1}} \to R^af_*\Q_{\Lambda_G}.
}
The lemma follows from the fact that $p^*$ is surjective with support on $S_b$. (Both assertions are checked fibrewise.) 
\end{proof}

Consider the Leray spectral sequence 
\eq{23}{E^{pq}_2 = H^p(\P^{g-1}, R^qf_*\Q_{\Lambda_G}) \Rightarrow H^{p+q}(\Lambda_G,\Q).
}
It follows from the lemma that $E_2^{pq} = H^p(S_{q/2},\Q(-q/2))$ (zero for $q$ odd) has weights $\le p+q$ with equality if either $p=0$ or $q\le 2(n-g-1)$. Since $E_s,\ s\ge 2$ is a subquotient of $E_2$ we get the same assertion for $E_s$. From the complex
\eq{}{E_s^{p-s,q+s-1} \to E_s^{pq}\to E_s^{p+s,q-s+1}
}
we deduce
\begin{prop} For the spectral sequence \eqref{23} we find in the range $q\le 2(n-g-1)$ or $p=0,\ q\le 2(n-g)$ that $E_\infty^{pq}= \Q(-(p+q)/2)$ if both $p,q$ are even, and $E_\infty^{pq}=(0)$ otherwise. In particular, the pullback $H^s(\P^{n-1}\times \P^{g-1},\Q) \to H^s(\Lambda_G,\Q)$ is an isomorphism for $s\le 2(n-g)$. 
\end{prop}

An interesting special case is that of {\it log divergent graphs}. By definition, $G$ is log divergent if $n=2g$. 
\begin{cor} For $G$ log divergent, $W_{n-3}H^{n-2}(X_G,\Q)$ dies in $H^{n-2}(\Lambda_G,\Q)$. 
\end{cor}

Concerning the motive of the graph hypersurface $X_G$ for an arbitrary graph $G$ we deduce
\begin{thm}$gr^W_aH^a(X_G,\Q)$ is pure Tate for any $a\in \Z$.
\end{thm}
\begin{proof}Consider the maps $\widetilde\Lambda_G \xrightarrow{\rho} \Lambda_G \xrightarrow{\pi} X_G$, where $\widetilde\Lambda_G$ is a resolution of singularities of $\Lambda_G$. By \cite{D}, Prop. 8.2.5, the image $\rho^*\pi^*H^a(X_G) \subset H^a(\widetilde\Lambda_G)$ is isomorphic to $gr^W_aH^a(X_G,\Q)$. This image is a subquotient of $H^a(\Lambda_G,\Q)$ which is mixed Tate by Prop. \ref{mixedtate}. Since the image is pure of weight $a$, the theorem follows. 
\end{proof}

\section{The Second Symanzik Polynomial for Graphs}\label{2symanzik}

The second Symanzik polynomial for a graph $G$ depends on masses and external momenta. More precisely, to each vertex $v$ one associates $p_v\in \R^D$ where $D$ is the dimension of space-time. The conservation of momentum condition is
\eq{}{\sum_v p_v=0.
}
In addition, to each edge $e$ is attached a mass $m_e \in \R$. The {\it propagator} $f_e$ associated to an edge $e$ is defined via the diagram
\eq{}{\begin{CD}H_1(G,\R^D) @>>> (\R^D)^E @>e^{\vee,2}-m_e^2>> \R \\
@. @VV \partial V \\
@. p\in(\R^D)^{V,0} 
\end{CD}
}
By definition $f_e$ is the function $e^{\vee,2}-m_e^2$ restricted to $\partial^{-1}(p) \cong H_1(G,\R^D)$. The amplitude is 
\eq{16}{A_G(p,m) := \int_{\partial^{-1}(p)} \frac{d^{Dg}x}{\prod_{e\in E} f_e}
}
Of course, $\partial^{-1}(p)$ can be identified (non-canonically) with $H_1(G, \R^D)$ and the integral can be viewed as an integral over $H_1(G, \R^D)$. 

Let $n=\# E(G)$. We consider edge variables $A_e$ as homogeneous coordinates on $\P^{n-1}$. (Sometimes it is convenient to order the variables and write $A_i$ rather than $A_e$.) Write 
\eq{}{\Omega := \sum (-1)^{i-1} A_idA_1\wedge\cdots \wedge\widehat{A_i}\wedge\cdots \wedge dA_n
}
For $F(A_e)$ homogeneous of degree $n$, the ratio $\Omega/F$ is a meromorphic form of top degree $n-1$ on $\P^{n-1}$. A chain of integration $\sigma$ is defined by
\eq{}{\sigma = \{(\dotsc,a_e,\dotsc)\ |\ a_e \ge 0\} \subset \P^{n-1}(\R)
}
A choice of ordering of the edges orients $\sigma$. 
\begin{ex}Suppose $G$ is a tree, i.e. $g=0$. Then $\partial^{-1}(p)$ is a point, and the integral simply becomes evaluation of $\frac{1}{\prod_e f_e}$ at this point. If, for example, $G$ is just a string with vertices $\{1,2,\dotsc,n\}$ and edges $(i,i+1),\ 1\le i\le n-1$, then with evident orientation we have $\partial(i,i+1)=(i+1)-(i)$. Write $p_i$ for the external momentum at the vertex $i$. We must find $q_{i,i+1}\in \R^D,\ 1\le i\le n$ such that $q_{i-1,i}-q_{i,i+1}=p_i, 1\le i\le n$ where $q_{0,1}=q_{n,n+1}=0$. This yields $q_{i,i+1}=-p_1-p_2-\cdots -p_i$ and the amplitude is
\ml{}{A_G(p,m)=\\
\frac{1}{(p_1^2-m_1^2)((p_1+p_2)^2-m_2^2)\cdots ((p_1+\dotsc+p_{n-1})^2-m_{n-1}^2)}.
}
\end{ex}
\begin{lem}[Schwinger parameters] Viewing the $f_i$ as independent coordinates, we have
\eq{}{\frac{1}{\prod_{i=1}^n f_i} = (n-1)! \int_\sigma \frac{\Omega}{(\sum A_if_i)^n}. 
}
\end{lem}
\begin{proof}In affine coordinates $a_i=A_i/A_n$ the assertion becomes
\eq{}{\frac{1}{\prod_{i=1}^n f_i} = (n-1)! \int_{0^{n-1}}^{\infty^{n-1}} \frac{da_1\cdots da_{n-1}}{(a_1f_1+\cdots +a_{n-1}f_{n-1}+f_n)^n}.
}
We have
\ml{}{d\Big(\frac{da_2\cdots da_{n-1}}{(a_1f_1+\cdots +a_{n-1}f_{n-1}+f_n)^{n-1}}\Big) = \\
-(n-1)f_1\frac{da_1\cdots da_{n-1}}{(a_1f_1+\cdots +a_{n-1}f_{n-1}+f_n)^n},
}
and the result follows by induction. 
\end{proof}

The amplitude \eqref{16} can thus be rewritten
\eq{}{A_G =\frac{1}{(n-1)!} \int_{\R^{Dg}}d^{Dg}x\int_\sigma\frac{\Omega}{(\sum A_ef_e)^n}. 
}

We would like to interchange the two integration operations. Following the physicists, we take the metric on $\R^D$ to be Euclidean. There is still an issue of convergence because the quadratic form is only positive semi-definite, but as mathematicians we are looking for interesting motives to study. The issue of convergence of a particular period integral is of secondary concern. 

We must evaluate 
\eq{24}{\int_{\R^{Dg}}\frac{d^{Dg}x}{(\sum A_ef_e)^n}.
}
To this end, we first complete the square for the quadratic form $\sum A_ef_e$. We identify $H_1(G,\R^D) = (\R^D)^g$ and write 
$$x_i=(x_i^1,\dotsc,x_i^D):H_1(G,\R^D) \to \R^D; \quad1\le i\le g;\quad \mathbf x = ({}^tx_1,\dotsc,{}^tx_g).
$$
Similarly, $\mathbf{p}=(\dotsc,{}^tp_v,\dotsc)_{v\neq v_0}$ where $p_v\in \R^D$ are the external momenta and we omit one external vertex $v_0$. 
We write
\eq{25}{\sum A_ef_e = \mathbf{x}M{}^t\mathbf{x} -2\mathbf{x}B{}\mathbf{p}+\mathbf{p}\Gamma {}^t\mathbf{p}-\mu
}
Here $M$ (resp. $B$, resp. $\Gamma$) is a $g\times g$ (resp. $g\times (\# V-1)$, resp. $(\# V-1)\times (\# V-1)$) matrix with entries which are linear forms in the $A_e$; and $\mu= \sum_e m_e^2A_e$. Note that $M$ is the symmetric $g\times g$ matrix associated to the configuration $H_1(G,\R)\subset \R^E$  as in \eqref{2}. In particular,$\psi_G = \det(M)$. Note also that the matrix operations in \eqref{25} are a bit exotic. Whenever column $D$-vectors in $\mathbf{x}$ and $\mathbf{p}$ are to be multiplied, the multiplication is given by the quadratic form on $\R^D$. 

To complete the square write $\mathbf{x} = \mathbf{x}' + M^{-1}B\mathbf{p}$ and $\mathbf{x}' = \mathbf{x}''R$ where $R$ is orthogonal with $RM{}^tR = \D$ diagonal. We get
\eq{26}{\sum A_ef_e =\mathbf{x}''\D{}^t\mathbf{x}''-(B\mathbf{p})(M^{-1}){}^t(B\mathbf{p})+\mathbf{p}\Gamma{}^t\mathbf{p}-\mu}

By definition, the second Symanzik polynomial is
\ml{27}{\phi_G(,\{A_e\},\mathbf{p},\{m_e\}) := \\
(B\mathbf{p})(\text{adj}(M)){}^t(B\mathbf{p})+(\mathbf{p}\Gamma{}^t\mathbf{p}-\mu)\psi_G(\{A_e\}).
}
Here $\text{adj}(M)$ is the adjoint matrix, so $M^{-1}=\text{adj}(M)\det(M)^{-1}$. We can rewrite \eqref{26}
\eq{28}{\sum A_ef_e =\mathbf{x}''\D{}^t\mathbf{x}''-\frac{\phi_G(A,p,m)}{\psi_G(A)}
}
Using the elementary identity
\eq{}{\int_{-\infty^N}^{\infty^N}\frac{du_1\cdots du_N}{(C_1u_1^2+\cdots + C_Nu_N^2)+L)^n} = \pi^N\prod_{i=1}^NC_i^{-1/2}L^{-n+(N/2)},
}
we now find (taking $N=gD$)
\ga{}{\int_{\R^{Dg}}\frac{d^{Dg}x}{(\sum A_ef_e)^n} = \frac{\pi^{Dg}\psi_G^{n-(g+1)D/2}}{\phi_G^{n-gD/2}}\\
A_G = \frac{\pi^{Dg}}{(n-1)!}\int_{\sigma\subset \P^{n-1}(\R)} \frac{\psi_G^{n-(g+1)D/2}\Omega}{\phi_G^{n-gD/2}}.
}
\begin{remark}Of particular interest is the {\it log divergent} case $D=2n/g$, when $A_G = \frac{\pi^{Dg}}{(n-1)!}\int_\sigma \frac{\Omega}{\psi_G^{D/2}}$ is independent of masses and external momenta. 
\end{remark}

To summarize, we now have three formulas for the amplitude
\ga{}{A_G = \int_{\R^{Dg}}\frac{d^{Dg}x}{\prod_{e\in E} f_e} \label{prodquads}\\
A_G =\frac{1}{(n-1)!} \int_{\R^{Dg}\times \sigma}\frac{d^{Dg}x\Omega}{(\sum A_ef_e)^n}\label{5.20}  \\
A_G = \frac{\pi^{Dg}}{(n-1)!}\int_{\sigma\subset \P^{n-1}(\R)} \frac{\psi_G^{n-(g+1)D/2}\Omega}{\phi_G^{n-gD/2}}.
}
To these, we add without proof a fourth (\cite{}, formula (6-89)) 
\eq{35a}{A_G = \frac{1}{(i(4\pi)^2)^g}\int_{\widetilde\sigma} \frac{\exp(i\phi_G/\psi_G)\prod_E dA_e}{\psi_G^{D/2}}. 
}
Here $\widetilde\sigma = [0,\infty]^E$ so $\sigma = \widetilde\sigma-\{0\}/\R^\times_+$. Philosophically, we can think of $\widetilde\sigma$ as the space of metrics (i.e. lengths of edges) on $G$. The integral then looks like a path integral on a space of metrics, with the action $\phi_G/\psi_G$. We will see in what sense this action is a limit of string theory action. 

Here is a useful way to think about the second Symanzik polynomial when the metric on space-time is euclidean. Let $G$ be a connected graph, and assume the metric on $\R^D$ is positive definite (i.e. euclidean). Then for $a_e>0$ the metric $\sum_e a_ee^{\vee,2}$ on $(\R^D)^E$ is positive definite as well, so there are induced metrics on $H_1(G, \R^D)$ and $(\R^D)^{V,0}$. For $p\in (\R^D)^{V,0}$ let $m_a(p)$ be the value of the metric.   
\begin{prop}We have $m_a(p)= \phi_G(a,p,0)/\psi_G(a)$. 
\end{prop} 
\begin{proof} The symmetric matrix $M$ above is positive definite when the edge coordinates $A_e>0$. It is then clear from \eqref{28} that the minimum of the metric in the fibre $\partial^{-1}(p)$ is given by $-\frac{\phi_G(A,p,0)}{\psi_G(A)}$. This is how the metric on the quotient is defined. (In general, if $V$ has a positive definite metric, the metric on a quotient $V/W$ is defined by identifying $V/W\cong W^\perp\subset V$.) 
\end{proof}

\section{Riemann Surfaces}\label{riemannsurfaces}

In sections \ref{riemannsurfaces} - \ref{limit} we will reinterpret formula \eqref{35a} for the amplitude. The graph $G$ becomes the dual graph of a singular rational curve $C_0$ of arithmetic genus $g$ which we view as lying at infinity on a moduli space of pointed curves of genus $g$. Vertices of $G$ correspond to irreducible components of $C_0$, and the external momentum associated to a vertex is interpreted in terms of  families of points meeting the given irreducible component of $C_0$. The chain of integration in \eqref{35a} is identified with the nilpotent orbit associated to the degenerating Hodge structures, and the action $\exp(i\phi_G/\psi_G)$ is shown to be a limit of actions involving heights. This is joint work with Jos\'e Burgos Gil, Javier Fresan, and Omid Amini.  

As a first step, in this section we will interpret the rank $1$ symmetric matrices $M_e$ on $H_1(G)$, \eqref{1}, in terms of the monodromy of the degenerating family of genus $g$ curves. We continue to assume $G$ is a connected graph with $g$ loops and $n$ edges. Stable rational curves $C_0$ associated to $G$ arise taking quotients of $\coprod_{V(G)} \P^1$ identifying a chosen point of $\P^1_v$ with a chosen point of $\P^1_w$ whenever there exists an edge $e$ with $\partial e=\{v,w\}$. For the moment we assume that every vertex of $G$ meets at least $3$ edges, so every $\P^1 \subset C_0$ has at least three ``distinguished'' points which are singularities of $C_0$. Note that if there is a vertex meeting $\ge 4$ edges, the corresponding $\P^1$ will have $\ge 4$ distinguished points and $C_0$ will have moduli. 
\begin{defn}Let $C = \bigcup \P^1$ be a curve obtained by identifying a finite set of pairs of points in $\coprod_V \P^1$ for some finite set $V$. The dual graph of $C$ is the graph with vertex set $V$ and edge set $E$ the set of pairs of points being identified. If $e\in E$ corresponds to $\{p_1,p_2\}$ with $p_i\in \P^1_{v_i}$ then the edge $e$ is taken to connect the vertices $v_1, v_2$. 
\end{defn}
\begin{ex}The dual graph of the curve $C_0$ constructed above is $G$. 
\end{ex}
\begin{prop}\label{genus} There is a canonical identification $H^1(C_0, \sO_{C_0}) \cong H^1(G,\Q)\cong H^1(C_0,\Q)$. In particular, the arithmetic genus of $C_0$ is equal to $g$, the loop number of $G$. 
\end{prop}
\begin{proof}Let $p: \coprod_{V(G)} \P^1 \to C_0$ be the identification map. We have an exact sequence of sheaves
\eq{34}{0 \to \sO_{C_0} \to p_*\sO_{\coprod \P^1} \to \sS \to 0
}
where $\sS$ is a skyscraper sheaf with stalk $k$ over each singular point. Since $p$ is a finite map, taking cohomology commutes with $p_*$ and we find
\eq{35}{0 \to k \to \bigoplus_{V(G)} k \xrightarrow{\delta} \bigoplus_{E(G)} k \to H^1(C_0, \sO_{C_0}) \to 0
}
It is straightforward to check that $\delta$ in the above can be identified with the dual to the boundary map calculating $H_1(G)$, so $\text{coker}(\delta) \cong H^1(G,k)$. The proof that $H^1(G,\Q)\cong H^1(C_0,\Q)$ is similar. One simply replaces the exact sequence of coherent sheaves \eqref{34} with an analogous sequence of constructible sheaves calculating Betti cohomology. 
\end{proof}
We recall some basic results about deformation theory for $C_0$, \cite{H}. There exists a smooth formal scheme $\widehat{S} = \text{Spf }k[[t_1,\dotsc,t_p]]$ and a (formal) family of curves $\widehat{\sC} \xrightarrow{\pi} \widehat{S}$ such that the fibre $\sC_0$ over $0\in \widehat{S}$ is identified with the curve $C_0$ above, and such that the family is in some sense maximal. In particular, $\widehat{\sC}$ is formally smooth over $k$ and the tangent space $T$ to $\widehat{S}$ at $0$ is identified with $\text{Ext}^1(\Omega^1_{C_0}, \sO_{C_0})$. \'Etale locally at the singular points $C_0 \cong \Spec k[x,y]/(xy)=: \Spec R$ so $\Omega^1_{C_0} \cong Rdx\oplus Rdy/(xdy+ydx)$. Thus $xdy \in \Omega^1_{C_0}$ is killed by both $x$ and $y$ and so $\Omega^1$ has a non-trivial torsion subsheaf supported at the singular points. The $5$-term exact sequence of low degree terms for the local to global Ext spectral sequence yields in this case a short exact sequence
\ml{36}{0 \to H^1(C_0, \underline{Hom}(\Omega^1_{C_0},\sO_{C_0})) \to \text{Ext}^1(\Omega^1_{C_0},\sO_{C_0}) \to \\
\Gamma(C_0, \underline{Ext}^1(\Omega^1_{C_0},\sO_{C_0})) \to 0.
}
The local ext sheaf on the right is easily calculated using the local presentation at the singular point as above
\eq{}{ 0 \to R \xrightarrow{1 \mapsto xdy+ydx} Rdx\oplus Rdy \to \Omega^1_R \to 0. 
}
One identifies in this way $\underline{Ext}^1(\Omega^1_{C_0},\sO_{C_0})$ with the skyscraper sheaf having one copy of $k$ supported at each singular point. The subspace $H^1(C_0,  \underline{Hom}(\Omega^1_{C_0},\sO_{C_0}))\subset T$ corresponds to deformations which keep all the double points, i.e. only the chosen points on the $\P^1_v$ move. In general, $r$ points on $\P^1$ have $r-3$ moduli, so 
\ga{}{\dim H^1(C_0,  \underline{Hom}(\Omega^1_{C_0},\sO_{C_0})) = \sum_{v\in V(G)} (\#\text{edges through $v$}-3) \\
\dim T = \sum_{v\in V(G)} (\#\text{edges through $v$}-3) +\#E(G) = \\
-3V(G) +3\# E(G)= 3h_1(G) - 3 = 3g_a(C_0)-3. \notag
} 
Here $g_a(C_0)$ is the arithmetic genus which coincides with the usual genus of a smooth deformation of $C_0$. The last identity follows from proposition \ref{genus}. Note $3g-3$ is the dimension of the moduli space of genus $g$ curves. 

The map $T=Tan_{\widehat{S},0} \surj \bigoplus_E k$ arises as follows. For each $e\in E$ there is a singular point $p_e \in C_0$. The deformations of $C_0$ which preserve the singularity at $p_e$ give a divisor $\widehat{D}_e \subset \widehat{S}$. Let $g_e \in \sO_{\widehat{S}}$ define $\widehat{D}_e$. The functional $T \to \bigoplus_E k \xrightarrow{pr_e} k$ is defined by $dg_e$. 
The geometric picture is then a collection of principal divisors $\widehat{D}_e\subset \widehat{S}$ meeting transversally. The subvariety cut out by the divisors is the locus of equisingular deformations of $C_0$ given by moving the singular points. If $G$ is trivalent, i.e. if every vertex of $G$ has exactly three adjacent edges, then $\{0\} = \bigcap \widehat{D}_e$. 

Deformation theory leads to a formal {\it versal} deformation $\widehat{\sC} \to \widehat{S}$, but these formal schemes can be spread out to yield an analytic deformation $\sC \to S$. Here $S$ is a polydisk of dimension $3g-3$. The divisors lift to analytic divisors $D_e:g_e=0$ on $S$. We fix a basepoint $s_0 \in S-\bigcup D_e$, and we wish to study the monodromy action on $H_1(C_{s_0},\Q)$. We choose simple loops $\ell_e\subset S-\bigcup_E D_e$ based at $s_0$ looping around $D_e$. We assume $\ell_e$ is contractible in $S-\bigcup_{\ve\neq e}D_\ve$.  The {\it Picard Lefschetz formula} gives the monodromy for the action of $\ell_*$ on $H_1(C_{s_0},\Q)$
\eq{42}{b \mapsto b+\langle b,a_e\rangle a_e
}
where $a_e\in H_1(C_{s_0},\Q)$ is the vanishing cycle associated to the double point on the curve which remains as we deform along $D_e$. 

A classical result in differential topology says that, possibly shrinking the polydisk $S$, the inclusion $C_0 \inj \sC$ admits a homotopy retraction $\sC \to C_0$ in such a way that the composition $\sC \to C_0 \to \sC$ is homotopic to the identity. It follows that $C_0 \inj \sC$ is a homotopy equivalence. In this way, one defines the specialization map
\eq{}{sp: H_1(C_{s_0},\Q) \to H_1(\sC,\Q) \cong H_1(C_0,\Q). 
}
\begin{lem}The specialization map $sp$ above is surjective. 
\end{lem}
\begin{proof} Intuitively, a loop in $H_1(C_0,\Q)$ can be broken up into segments connecting double points of the curve. These double points arise from shrinking vanishing cycles on $C_{S_0}$ so the segments can be modeled by segments in $C_{s_0}$ connecting the vanishing cycles. These segments connect to yield a loop in $C_{s_0}$ which specializes to the given loop in $C_0$. (One can give a more formal proof based on the Clemens-Schmid exact sequence, \cite{PS}.)
\end{proof}
\begin{lem}\label{lem_isotropic} The subspace $A\subset H_1(C_{s_0},\Q)$ spanned by the vanishing cycles $a_e$ is maximal isotropic.
\end{lem}
\begin{proof}As the base point $s_0$ approaches $0\in S$, the various $a_e$ approach the singular points $p_e\in C_0$. In particular, if $s_0$ is taken close to $0$, the $a_e$ are physically disjoint, so $\langle a_e, a_{e'}\rangle = 0$. Since the pairing on $H_1$ is symplectic, one has $\langle a_e,a_e\rangle=0$, so the subspace $A$ spanned by the vanishing cycles is isotropic. To see it is maximal, note we can express $C_0$ as a topological colimit $C_0 = C_{s_0}/\coprod S^1$ so we get an exact sequence $A \to H_1(C_{s_0}) \xrightarrow{sp} H_1(C_0)$. In particular,   
\eq{}{g=\dim H_1(C_0) \ge \dim H_1(C_{s_0}) - \dim A = 2g -\dim A. 
}
It follows that $\dim A \ge g$ so $A$ is maximal isotropic. In terms of a symplectic basis $a_1,\dotsc,a_g,b_1,\dotsc,b_g$ we write
\eq{45a}{a_e = \sum_{i=1}^g c_{e,i}a_i.
}
\end{proof}

The link between the combinatorics of the graph polynomial and the monodromy is given by the following proposition. Write $N_e = \ell_e-id$ so by \eqref{42} we have $N_e(b)=\langle b,a_e\rangle a_e$. By lemma \ref{lem_isotropic} we get $N_e^2=0$ so $N_e=\log(\ell_e)$. We consider the composition
\eq{45}{H_1(G) \cong H_1(C_0) \cong H_1(C_{s_0})/A \xrightarrow{N_e} A \cong (H_1(C_{s_0})/A)^\vee \cong H_1(G)^\vee
}
\begin{prop}The bilinear form on $H_1(G)$ given by \eqref{45} coincides with the bilinear form $M_e$ in \eqref{1}. 
\end{prop}
\begin{proof}Let $b\in H_1(C_{s_0})$. We can identify $sp(b) \in H_1(C_0)\cong H_1(G)$ with a loop $\sum_e n_ee$. Here $n_e= \langle b,a_e\rangle$ is the multiplicity of intersection of $b$ with the vanishing cycle $a_e$. The quadratic form on $H_1(C_{s_0})$ corresponding to $N_e$ sends $b\mapsto  \langle b, \langle b,a_e\rangle a_e\rangle = n_e^2$. The quadratic form on $H_1(G)$ corresponding to $M_e$ maps the loop $\sum n_\ve \ve$ to $n_e^2$. 
\end{proof}
\begin{remark}\label{rmk25} In terms of the basis $b_i$ for $B\cong H_1(G)$, we can write $M_e=(c_{e,i}c_{e,j})$ using the notation of \eqref{45a}. We will generalize this to relate the monodromy for punctured curves to the combinatorics of the second Symanzik in \eqref{61} through \eqref{65}. 
\end{remark}
\section{Biextensions and Heights}

Our objective in this section will be to link the second Symanzik polynomial (definition \ref{second_symanzik}) to geometry. Recall for a graph $G$ with edges $E$ and vertices $V$, the second Symanzik $\phi(H,w,\{A_e\})$ depends on $H=H_1(G)\subset \Q^E$ and on $w\in \R^{V,0}$. (We will see later how to extend the construction and take $w\in (\R^D)^{V,0}$ where $D$ is the dimension of space-time.) Recall that $\phi$ is quadratic in $w$. We will work with the corresponding bilinear function
\ml{7.1}{\phi(H,w,w',\{A_e\}):= \\
\phi(H,w+w',\{A_e\})-\phi(H,w,\{A_e\})-\phi(H,w',\{A_e\}).
}

We change notation and assume $\sC \xrightarrow{\pi} S$ is a family of pointed curves. More precisely, we suppose given two collections $\sigma_{v,i}: S \to \sC,\ v\in V,\ i=1,2$ of sections of $\pi$. We assume $\sigma_{v,i}(S)\cap C_0\in \P^1_v$.  Since $\sC$ is taken to be regular over the ground field, the sections cannot pass through double points of $C_0$. 
We assume further that the $\sigma_{v,1}$ and $\sigma_{v,2}$ are disjoint. It follows after possibly shrinking $S$ that the multi-sections ${\sigma_1}$ and $\sigma_2$ are disjoint as well. 
Let $W,(*,*)$ be an $\R$-vector space with a symmetric quadratic form. (In fact, for us $W=\R$ or $W=\R^D$ with the Minkowski metric.) We fix $W$-divisors 
\eq{}{\frak A_i := \sum r_{v,i}\sigma_{v,i};\quad r_{v,i} \in W,\ \sum_v r_{v,i}=0,\ i=1,2.
}

More generally, we should work with a diagram
\eq{}{\begin{CD}T_1\amalg T_2 @>\text{closed immersion} >> \sC \\
@| @VVV \\
T_1\amalg T_2 @>\tau_1\amalg \tau_2 >>  S
\end{CD}
}
where $\tau_i$ are finite \'etale. The labels $r_{v,i}$ would be replaced by sections of local systems of $\R$-vector spaces $R_i$ on $T_i$ equipped with trace maps $R_i \surj \tau_i^*\R_S$ so we can talk about sections of degree $0$. We have no use for this generalization, but we mention it exists. 

The $W$-divisors will play the role of external momenta. We need to define the action which we write 
\eq{}{\mathbf{S}[C_{s_0},\frak A_1,\frak A_2]
} 
and which will tend to $\phi(H_1(G),\sum r_{v,1}v,\sum r_{v,2}v)$ as $s_0 \to 0$. This action is given by the archimedean height pairing which is defined as follows. We identify the section ${\sigma_1}$ on $C_{s_0}$ with its image ${\sigma_1}\subset C_{s_0}$. We consider the exact sequence of Hodge structures
\eq{}{0 \to H^1(C_{s_0},\Z(1)) \to H^1(C_{s_0}-{\sigma_1},\Z(1)) \xrightarrow{res} \coprod^0_{{\sigma_1}}\Z \to 0. 
}
This sequence is canonically split as an exact sequence of $\R$-Hodge structures. I.e. there exists a canonical splitting
\eq{}{\coprod_{{\sigma_1}}^0\R \to F^0(H^1(C_{s_0}-{\sigma_1})(1))\cap H^1(C_{s_0}-{\sigma_1},\R(1));\quad \frak A \mapsto \omega_{\frak A}.
}
Here $\omega_{\frak A} \in \Gamma(C_{s_0},\Omega^1(\log \sigma_1))$. For $\frak B$ another $\R$-divisor of degree $0$ on $C_{s_0}$ with support away from the support of $\frak A$, we can write $\frak B = \partial \beta$ where $\beta$ is a $1$-chain on $C_{s_0}-\text{Supp}(\frak A)$ and define the archimedean height
\eq{height}{\langle \frak A, \frak B\rangle := \int_{\beta} \omega_{\frak A} \in \C/i\R = \R. 
}
Because $\omega_{\frak A}$ is an $\R(1)$-class, changing $\beta$ by a class in $H_1(C_0,\R)$ does not change the real part of the integral so the height is well-defined.  

Finally, we can couple this pairing to the quadratic form on $W$ and define $\langle\frak A_1,\frak A_2\rangle$. Indeed, if we choose basepoints $\sigma_{0,i} \in \sC_{s_0}$ we can write $\sum_v r_{v,i}\sigma_{v,i} = \sum_v r_{v,i}(\sigma_{v,i}-\sigma_{0,i})$ and define
\eq{7.8}{\mathbf{S}[C_{s_0},\frak A_1,\frak A_2]= \langle\frak A_1,\frak A_2\rangle = \sum_{v,v'}(r_{v,1},r_{v',2})\langle\sigma_{v,1}-\sigma_{0,1},\sigma_{v',2}-\sigma_{0,2}\rangle.
}
This is independent of the choice of $\sigma_{0,i}$. Changing e.g. $\sigma_{0,1}$ to $\sigma_{0,1}{}'$ changes the above by
$$\sum_{v,v'}(r_{v,1},r_{v',2})\langle\sigma_{0,1}{}'-\sigma_{0,1},\sigma_{v',2}-\sigma_{0,2}\rangle
$$
which vanishes because $\sum_v r_{v,1}=0$.

\section{The Poincar\'e Bundle}
To understand the behavior of the height in a degenerating family, it is convenient to use the Poincar\'e bundle. We first recall the Poincar\'e bundle for a single compact complex torus $T:=V/\Lambda$, ($V$ finite dimensional $\C$-vector space, $\Lambda \subset \V$ a cocompact lattice.) Let $\widehat{V}:= \text{Hom}_{\overline{\C}}(V,\C)$ be the $\C$-vector space of $\C$-antilinear functionals on $V$, and let $\widehat{\Lambda}:= \{\phi \in \widehat{V}\ |\ \phi(\Lambda) \subset \Z\}$. By definition, the dual torus $\widehat{T}:= \widehat{V}/\widehat{\Lambda}$. 

The Poincar\'e bundle is a $\C^\times$-bundle $\sP^\times$ on $T\times \widehat{T}$. (NB. It will be more convenient to work with the principal $\G_m$ or $\C^\times$-bundle rather than the corresponding line bundle.) It is characterized by two properties:\nn
(i)$\sP^\times|_{\{0\}\times \widehat{T}}$ is trivial with a given trivialization. \nn
(ii) $\sP^\times|_{T\times \{\phi\}} = L_\phi$;
where $L_\phi$ is the $\C^\times$-bundle on $T$ associated to the representation of the fundamental group
\eq{}{\pi_1(T)=\Lambda \subset V \xrightarrow{\phi} \C \xrightarrow{\exp(2\pi i\cdot)} \C^\times.
}

To construct the Poincar\'e bundle over the family of all principally polarized abelian varieties of dimension $g$, we recall the Siegel domain 
\eq{}{\H_g := \{M\text{ $g\times g$ complex symmetric matrix }|\ Im(M)>0\}.
}
(The following description of the Poincar\'e bundle is taken from an unpublished manuscript of J. Burgos Gil. Details are omitted here. They will appear in a forthcoming paper with Burgos Gil, Fresan and Amini.) The group $Sp_{2g}(\R)$ acts on $\H_g$ by $\begin{pmatrix}A & B \\ C & D\end{pmatrix}M = (AM+B)(CM+D)^{-1}$. The quotient $\sA_g:= \H_g/Sp_{2g}(\Z)$ is the Siegel moduli space parametrizing principally polarized abelian varieties of dimension $g$.  We can also think of $\sA_g$ as parametrizing polarized Hodge structures of weight $1$. To $\Omega\in \H_g$ we associate the map $H_\Z:= \Z^{2g} \to \C^g$ given by the $2g\times g$ matrix $\begin{pmatrix}\Omega \\ I_g\end{pmatrix}$. The symplectic form on $H_\Z$ is given by $\begin{pmatrix}0 & I_g \\ -I_g & 0\end{pmatrix}$. 

The space 
\eq{}{\widetilde{X}:= \H_g \times \text{Row}_g(\C) \times \text{Col}_g(\C) \times \C
} 
is a homogeneous space for the group
\ml{}{\widetilde{G} = \Big\{\begin{pmatrix} 1 & \lambda_1 & \lambda_2 & \alpha \\0 & A & B & \mu_1 \\ 0 & C & D & \mu_2 \\ 0 & 0 & 0 & 1\end{pmatrix}\Big | \lambda_i \in \text{Row}_g(\R),\ \mu_j \in \text{Col}_g(\R), \alpha \in \C,\\ \begin{pmatrix}A & B \\ C & D\end{pmatrix}\in Sp_{2g}(\R)\Big\}.
}
The action is determined by formulas
\ml{}{\begin{pmatrix} 1 & 0 & 0 & \alpha \\0 & A & B & 0 \\ 0 & C & D & 0 \\ 0 & 0 & 0 & 1\end{pmatrix}(\Omega,W,Z,\rho) = \\
((A\Omega+B)(C\Omega+D)^{-1},W(C\Omega+D)^{-1},(C\Omega+D)^{-t}Z,\rho-WC^t(C\Omega+D)^{-t}Z)
}
\eq{}{\begin{pmatrix} 1 & \lambda_1 & \lambda_2 & 0 \\0 & I & 0 & 0 \\ 0 & 0 & I & 0 \\ 0 & 0 & 0 & 1\end{pmatrix}(\Omega,W,Z,\rho) = 
(\Omega, W+\lambda_1\Omega+\lambda_2,Z,\rho+\lambda_1Z)
}
\eq{}{\begin{pmatrix} 1 & 0 & 0 & 0 \\0 & I & 0 & \mu_1 \\ 0 & 0 & I & \mu_2 \\ 0 & 0 & 0 & 1\end{pmatrix}(\Omega,W,Z,\rho) = 
(\Omega,W,Z+\mu_1-\Omega\mu_2,\rho-W\mu_2)
}
\eq{}{\begin{pmatrix} 1 & 0 & 0 & \alpha \\0 & I & 0 &0 \\ 0 & 0 & I &0 \\ 0 & 0 & 0 & 1\end{pmatrix}(\Omega,W,Z,\rho) = 
(\Omega,W,Z,\rho+\alpha).
}

Write $\widetilde{G}(\Z)\subset \widetilde{G}$ for the subgroup with entries in $\Z$. 
\begin{thm}The quotient
\eq{}{\widetilde{G}(\Z)\backslash(\H_g\times \text{Row}_g(\C)\times \text{Col}_g(\C)\times \C)
}
is the dual of the Poincar\'e bundle $\sP^\times$. 
\end{thm}
\begin{proof} Omitted.
\end{proof}
\begin{thm}The Poincar\'e bundle admits a translation-invariant metric $\log||\cdot||:\sP^\times \to \R$. For $(\Omega,W,Z,\alpha) \in \H_g\times \text{Row}_g(\C)\times \text{Col}_g(\C)\times \C)$, we have 
\eq{64}{\log||(\Omega,W,Z,\alpha)|| = 
4\pi\Big(Im(\alpha)+Im(W)(Im\Omega)^{-1}Im(Z)\Big).
}
\end{thm}
\begin{proof}Omitted. 
\end{proof}

The Poincar\'e bundle has a Hodge-theoretic interpretation as the moduli space for biextensions, which are mixed Hodge structures $M$ with weights $-2, -1, 0$. We assume 
\eq{}{W_{-2}M=\Z(1),\ gr^W_{-1}M = H,\ gr^W_0M=\Z(0),
} 
where $H$ is a rank $g$ principally polarized Hodge structure of weight $-1$.  To see this, we remark that $\Lambda_\C = H_1(T,\C)$ has a Hodge structure of weight $-1$ with $F^{0}\Lambda_\C := \ker(\Lambda_\C \surj V)$ and $gr^{-1}_F\Lambda_\C = V$. For $0\neq t \in T$ the relative homology $H_1(T,\{0,t\},\Z)$ yields an extension of Hodge structures
\eq{}{0 \to H_1(T,\Z) \to H_1(T,\{t,0\}) \to \Z(0) \to 0. 
}
A point $\phi \in \widehat{T}=Hom_{\overline{\C}}(V,\C)/\widehat{\Lambda}$ yields a character $\exp(2\pi i\phi): \Lambda \to \C^\times$ and hence a principal $\C^\times$-bundle $L_\phi^\times$ over $T$. The corresponding sequence of homology groups looks like
\eq{}{0 \to \Z(1) \to H_1(L_\phi^\times,\Z) \to H_1(T,\Z) \to 0
}
It follows from exactness of the Hodge filtration functor that $F^0H_1(L_\phi^\times,\C) = F^0H_1(T,\C)$. Thus
\eq{}{L_\phi^\times \cong gr^{-1}_F H_1(L_\phi^\times,\C)/H_1(L_\phi^\times,\Z) = \text{Ext}^1_{MHS}(\Z(0),H_1(L_\phi^\times,\Z)).
}
The projection $L_\phi^\times \surj T$ yields a diagram of Hodge structures coming from $\ell \in L_\phi^\times$ lying over $t\in T$ \minCDarrowwidth1cm
\eq{69}{\begin{CD}@. @. 0 @. 0 \\
@. @. @VVV @VVV \\
0 @>>> \Z(1) @>>> H_1(L_\phi^\times, \Z) @>>> H_1(T,\Z) @>>> 0 \\
@. @| @VVV @VVV \\
0 @>>> \Z(1) @>>> M_\ell @>>> M_t @>>> 0 \\
@. @. @VVV @VVV \\
@. @. \Z(0) @= \Z(0) \\
@. @. @VVV @VVV \\
@. @. 0 @. 0
\end{CD}
}
The biextension corresponding to $\ell \in L_\phi^\times = \sP|_{T\times \{\phi\}}$ is then $M_\ell$ with $W_{-1}M_\ell = H_1(L_\phi^\times, \Z)$ and $M_\ell/W_{-2} = M_t$. 

Let $\mathfrak A, \mathfrak B$ be divisors of degree $0$ on a smooth curve $C$. Assume the supports $|\mathfrak A|$ and $|\mathfrak B|$ are disjoint. We associate to $\mathfrak A, \mathfrak B$ a biextension Hodge structure $H^1(C-\mathfrak A, \mathfrak B; \Z)$. (Notation like  $H^1(C-\mathfrak A, \mathfrak B; \Z)$ or $H^1(C, \mathfrak B; \Z)$ is of course abusive. These are subquotients of $H^1(C-|\mathfrak A|, |\mathfrak B|; \Z)$ and $H^1(C, |\mathfrak B|; \Z)$ which fit into a diagram \minCDarrowwidth.5cm
\eq{70}{\begin{CD}@. @. 0 @. 0 \\
@. @. @VVV @VVV \\
0 @>>> \Z(1) @>>> H^1(C, \mathfrak B;\Z(1)) @>>> H^1(C, \Z(1)) @>>> 0 \\
@. @| @VVV @VVV \\
0 @>>> \Z(1) @>>> H^1(C-\mathfrak A, \mathfrak B; \Z(1)) @>>> H^1(C-\mathfrak A, \Z(1)) @>>> 0 \\
@. @. @VVV @VVV \\
@. @. \Z(0) @= \Z(0) \\
@. @. @VVV @VVV \\
@. @. 0 @. 0.)
\end{CD}
}

Diagrams \eqref{69} and \eqref{70} are related as follows. Take $T=J(C)$ and let $t\in T$ be the image of a $0$-cycle $\mathfrak A$ on $C$ of degree $0$. The bundle $L_\phi^\times$ is a group which is an extension of $T$ by $\C^\times$. We view $C$ as embedded in $T=J(C)$. There exists a $0$-cycle $\mathfrak B$ on $C$ of degree $0$ such that $L_\phi^\times|_C \cong \sO_C(\mathfrak B)-\text{$0$-section}$. Then $L_\phi^\times|_{C-|\mathfrak B|} \cong \C^\times \times (C-|\mathfrak B|)$, and this trivialization is canonical upto $c\in \C^\times$. Fix such a trivialization $\mu: C-|\mathfrak B| \to L_\phi^\times$. Write $\mathfrak A = \sum n_ia_i$, and let $\ell = \sum \mu(a_i)^{n_i}$. Then $\ell \in L_\phi^\times$ is well-defined independent of the choice of $\mu$, and diagram \eqref{69} for $\ell\in L_\phi^\times$ coincides with \eqref{70} for $\mathfrak A$ and $\mathfrak B$. 

Note that $L_\phi^\times$ is an abelian Lie group. As such, it has a unique maximal compact subgroup $K$, and 
\eq{}{L_\phi^\times/K \cong \C^\times/S^1 \cong \R.
}
Let $\rho_\phi: L_\phi^\times \to \R$ be the induced map.

\begin{prop} With notation as above, the following quantities are equal\nn
(i) $\langle\mathfrak A,\mathfrak B\rangle$.\nn
(ii) $\log ||H^1(C-\mathfrak A, \mathfrak B; \Z(1))||$.\nn
(iii) $\rho_\phi(\ell)$. 
\end{prop}
\begin{proof} Omitted.
\end{proof}

\section{The main theorem; nilpotent orbit and passage to the limit}\label{limit}

Write $U^*=(\Delta^*)^{E}\times \Delta^{3g-3-E+N}$, where $N$ parameters are added to accommodate markings associated to external momenta. Consider our family $\sC^* \xrightarrow{\pi} U^*$. Let $\widetilde{U} = \mathbb H^E \times \Delta^{3g-3-E+N} \to U^*$ be the universal cover with coordinate map $t_e = \exp(2\pi iz_e)$. We fix a basepoint $t_0 \in U^*$. We choose a symplectic basis $a_1,\dotsc,a_g,b_1,\dotsc,b_g$ in $H_1(C_{t_0},\Z) = A\oplus B$ with the vanishing cycles $a_e\in A$. We fix divisors $\delta=\delta_t, \mu=\mu_t$ which have degree $0$ on each $C_t$ and are disjoint. We choose $1$-chains $\sigma_{\delta,t}$ and $\sigma_{\mu,t}$ on $C_t$ which have $\delta_t$ and $\mu_t$ as boundary.   We fix a basis $\{\omega_{i,t}\}$ of the holomorphic $1$-forms on $C_t$ such that $\int_{a_i}\omega_j = \delta_{ij}$. The classical period matrix is then $(\int_{b_i}\omega_{j,t})$. Finally, we choose $1$-forms of the third kind $\omega_{\delta,t}, \omega_{\mu,t}$ with log poles on $\delta_t$ and $\mu_t$ respectively. 

We define the period map $\widetilde\phi: \widetilde U \to \mathbb H_g\times \C^g\times \C^g\times \C$
\eq{51}{\widetilde\phi(z) = \Big((\int_{b_i}\omega_{j,t})_{i,j}, (\int_{\sigma_{\delta,t}}\omega_{j,t})_j,(\int_{\sigma_{\mu,t}}\omega_{j,t})_j,\int_{\sigma_{\delta,t}}\omega_{\mu,t}\Big)
}

We check that this is the correct map as follows. It suffices to work pointwise, so we drop the $t$. We identify $J(C) = \text{Ext}^1_{MHS}( H^1(C, \Z),\Z(0))$ so e.g. $\mu$ defines an extension 
\eq{}{0 \to \Z(0) \to H_\mu \to H^1(C, \Z) \to 0.
}
Dually, we can think of $\delta$ as defining an extension of $\Z(0)$ by $H^1(C, \Z(1))$. We want to think of $\delta$ as defining an extension of $\Z(0)$ by $H_\mu(1)$ which ``lifts'' this extension. The diagram looks like
\eq{}{\begin{CD}@. 0 @. 0 \\
@. @VVV @VVV \\
@. \Z(1) @= \Z(1) \\
@. @VVV @VVV \\
0 @>>> (H_\delta)^\vee(1) @>>> M_{\mu,\delta} @>>> \Z(0) @>>> 0 \\
@. @VVV @VVV @| \\
0 @>>> H^1(C, \Z(1)) @>>> H_\mu @>>> \Z(0) @>>> 0 \\
@. @VVV @VVV \\
@. 0 @. 0.
\end{CD}
}
The exact sequence of ext groups is
\ml{55}{0 \to \text{Ext}^1_{MHS}(\Z(0), \Z(1)) \to \text{Ext}^1_{MHS}( \Z(0), H_\delta^\vee(1)) \to \\
\text{Ext}^1_{MHS}(\Z(0), H^1(C,\Z(1))) \to 0
} 
More concretely, we can think of $\int_{\sigma_\mu}$ as a functional on $\Gamma(C, \Omega^1) = F^1H^1(C, \C)$. Let $\Gamma(C, \Omega^1(\log \delta))\supset \Gamma(C, \Omega^1)$ be spanned by the $\omega_j$ together with the form of the third kind $\omega_\delta$. We can define $H_1(C-\delta, \Z)$ as a quotient of $H_1(C-\text{Supp}(\delta),\Z)$ in such a way that we get an exact sequence
\ml{56}{0 \to \C/\Z \to \Gamma(C,\Omega^1(\log \delta))^\vee/H_1(C-\delta,\Z) \to \\
\Gamma(C,\Omega^1)^\vee/H_1(C,\Z) \to  0
}
The exact sequences \eqref{55} and \eqref{56} coincide. The description of $\widetilde\psi$ in \eqref{51} follows from this; lifting $\int_{\sigma_\mu}$ to a functional on $ \Gamma(C,\Omega^1(\log \delta))$. Indeed, $(\int_{\sigma_{\delta,t}}\omega_{j,t})_j$ and $(\int_{\sigma_{\mu,t}}\omega_{j,t})_j$ yield the points in $J(C)$ associated to $\delta_t$ and $\mu_t$ respectively. The term $ \Gamma(C,\Omega^1(\log \delta))^\vee/H_1(C-\delta,\Z)$ in \eqref{56} shows that the class of $\delta$ in the generalized jacobian extension corresponding to $\mu$ is obtained by taking $\int_{\sigma_{\delta,t}}\omega_{\mu,t}$ as entry in $\C$. 

The idea will be to use the nilpotent orbit theorem to understand the limiting behavior of \eqref{51}. This theorem in the case of variations of polarized pure Hodge structures was proven by W. Schmid, \cite{S}. We need the more general case of a variation of mixed Hodge structures, \cite{P},\cite{KNU},\cite{KNUII},\cite{KNUIII}. In our case, we have the diagram
\eq{}{\begin{CD} \widetilde{U} @>\tilde\psi>> \H_g\times \text{Row}_g(\C)\times \text{Col}_g(\C)\times \C \\
@VVV @VVV \\
U^* @>>> \widetilde{G}(\Z)\backslash\Big(\H_g\times \text{Row}_g(\C)\times \text{Col}_g(\C)\times \C\Big)
\end{CD}
}
The action of the fundamental group $\Z^E$ of $U^*$ is unipotent, and we write $\sN_e$ for the logarithm of the generator $1_e \in \Z^E$. The expression $\tilde\psi(z):=\exp(-\sum_E z_e\sN_e)\tilde\phi(z)$ takes values in a compact dual $\check\sM$ which is essentially a flag variety parametrizing filtrations $F^*\C^{g+2}$ satisfying the conditions to be the Hodge filtration on a biextension of genus $g$. It is invariant under $z_e \mapsto z_e+1$ and so descends to a map $\psi: U^* \to \check\sM$. 
\begin{thm}\label{nilpotentorbit} [Nilpotent orbit theorem] (i) The limit 
$$F_\infty := \lim_{s\to 0} \psi(s) \in \check\sM
$$ 
exists. \nn
(ii) We have
$$\exp(\sum_E z_e\sN_e)F_\infty \in \H_g\times \text{Row}_g(\C)\times \text{Col}_g(\C)\times \C$$ 
whenever $Im(z_e)>>0, \forall e.$ \nn
(iii) For an invariant metric $d$ on $\H_g\times \text{Row}_g(\C)\times \text{Col}_g(\C)\times \C$ there exist constants $b, K$ such that for $\min_e(Im(z_e))>>0$ we have
$$d(\tilde\phi(z),\exp(\sum_E z_e\sN_e)F_\infty) \le K(\min_e(Im(z_e))^b \exp(-2\pi \min_e(Im( z_e))).
$$
\end{thm}

We want to understand how the logarithm of monodromy maps $N_e$ act on the entries in \eqref{51}. We have using \eqref{45a} and $\int_{a_i}\omega_j = \delta_{ij}$
\eq{57}{N_e(\int_{b_i}\omega_{j,t}) = \langle b_i,a_e\rangle\int_{a_e}\omega_{j,t} = c_{e,i}c_{e,j}.
}
By remark \ref{rmk25} we conclude
\eq{58}{N_e(\int_{b_i}\omega_{j,t})_{i,j} = M_e.
}
Said another way, if we view $b_i, b_j \in H_1(G)$, we have 
\eq{59}{N_e(\int_{b_i}\omega_{j,t})_{i,j} = (\dotsc, b_iM_eb_j^t\ldots)
}
To define the Picard-Lefschetz transformation on the path $\sigma_\delta\in H_1(C,\delta)$ we define vanishing cycles $a_e^\delta \in H_1(C-|\delta|)$. With this refinement, the P.-L. formula
\eq{60}{N_e(\sigma_\delta) = \langle\sigma_\delta,a_e^\delta\rangle a_e
}
holds. Here $N_e$ is viewed as an endomorphism of $H_1(C,\delta)$. Thus
\eq{61}{N_e(\int_{\sigma_{\delta,t}}\omega_{j,t})_j = (\langle\sigma_\delta,a_e^\delta\rangle c_{e,j})_j
}
To understand this, notice that the specialization identification $sp: B\cong H_1(G)\subset \R^E$ maps $b_j \mapsto \sum_e \langle b_j, a_e\rangle e= \sum_e c_{e,j}e$. Note that $ \langle b_j, a_e\rangle$ counts how many times (with orientation) the closed chain $sp(b_j)$ passes through the singular point $pt_e \in C_0$. Similarly we can think of $sp(\sigma_\delta)$ as a $1$-chain on $C_0$. It is not closed, but it bounds the $0$-chain corresponding to $\delta$. Thus, it corresponds to $\sum_e \langle \sigma_\delta, a^\delta_e\rangle e \in \R^E$.  In this way we can embed $H_1(C_s,\delta_s)/A \inj \R^E$ with the path $\sigma_\delta$ mapping as indicated. Formula \eqref{61} then can be interpreted as the pairing applied to $\sigma_\delta$ and $b_j$. 

Finally we have 
\eq{62}{N_e(\int_{\sigma_{\delta,t}}\omega_{\mu,t}) = \langle \sigma_{\delta,t},a^\delta_e\rangle \int_{a_e^\mu}\omega_{\mu}
}
Notice that $\omega_\mu \in H^1(C_t-\mu_t) \cong H_1(C_t,\mu_t)$. Modifying by a linear combination of the $\omega_j$ we can assume under that identification that $\omega_\mu$ corresponds to $\sigma_\mu$. Thus, we have
\eq{63}{N_e(\int_{\sigma_{\delta,t}}\omega_{\mu,t}) = \langle \sigma_{\delta,t},a^\delta_e\rangle  \langle \sigma_{\mu,t},a^\mu_e\rangle 
}

We can summerize our computations as follows. Consider the relative homology $H_1(C_t,\delta_t,\mu_t)$ which contains classes for the two paths $\sigma_{\delta,t}$ and $\sigma_{\mu,t}$. Let $A_{\sigma,\mu} \subset H_1(C_t,\delta_t,\mu_t)$ be the kernel of the specialization map 
\eq{84}{sp: H_1(C_t,\delta_t,\mu_t)/A_{\sigma,\mu} \to H_1(C_0,\delta_0,\mu_0) \subset \R^E. 
}
We can now consider the composition 
 \ml{65}{H_1(C_t,\delta_t,\mu_t)/A_{\sigma,\mu} \xrightarrow{sp} H_1(C_0,\delta_0,\mu_0) \subset \R^E \\
 \xrightarrow{(\dotsc,y_e,\ldots)} \R^E \to (H_1(C_t,\delta_t,\mu_t)/A_{\sigma,\mu})^\vee
 }
Replacing $N_e$ by $\sum y_eN_e$, the formulas \eqref{58}, \eqref{61}, and \eqref{63} give the pairings $\langle b_i,b_j\rangle, \langle \sigma_\delta, b_i\rangle, \langle \sigma_\delta,\sigma_\mu\rangle$ for the pairing defined by \eqref{65}.  

In our case, the nilpotent orbit theorem \ref{nilpotentorbit} says that the family \eqref{51} is well-approximated by 
\ml{}{\Big(\Omega_\infty+\sum_E z_eM_e,W_\infty+\sum_E z_e(\dotsc,\langle \sigma_\delta,a_e^\delta\rangle c_{ej},\ldots),\\
 Z_\infty+\sum_E z_e(\dotsc,\langle \sigma_\mu,a_e^\mu\rangle c_{ej},\ldots), \alpha_\infty+\sum_E z_e\langle \sigma_{\delta,t},a^\delta_e\rangle  \langle \sigma_{\mu,t},a^\mu_e\rangle \Big).
}
We write $z_e=x_e+iy_e$ and $t_e=\exp(2\pi iz_e)$ so $y_e=-(\log|t_e|)/2\pi$. We imagine 
\eq{}{y_e=Y_e/\alpha'
}
where $Y_e>0$ is constant and the {\it string tension} $\alpha' \to 0$. Substituting from \eqref{64} yields the following expression for the height
\ml{}{\frac{1}{\alpha'}\Big[(\alpha' ImW_0+\sum_eY_e(\dotsc,\langle\sigma_\delta,a_e^\delta\rangle c_{ej}\ldots))\cdot \\
\cdot(\alpha' Im\Omega_0+\sum_eY_eM_e)^{-1}(\alpha' ImZ_0+\sum_e Y_e(\dotsc,\langle\sigma_\mu,a_e^\mu\rangle c_{ej}\ldots))\\
 -\alpha' Im \alpha_0-\sum_e Y_e\langle \sigma_{\delta},a^\delta_e\rangle  \langle \sigma_{\mu},a^\mu_e\rangle\Big]
}

Now, using the determinant formula for the second Symanzik polynomial (remark \ref{det_formula}) we conclude
\begin{thm}Let $\tilde\phi:\widetilde{U} \to \sP$ be as in \eqref{51}. Then
\eq{}{\lim_{\alpha'\to 0} \alpha'\log||\tilde\phi(z)|| = \phi_G(p,p',\{Y_e\})/\psi_G(\{Y_e\}).
}
Here $z_e=X_e+i\frac{Y_e}{\alpha'}$, $\psi$ and $\phi$ are the first and second Symanzik polynomials, and $p,p'$ denote external momenta (see \eqref{7.1}).
\end{thm}
\begin{proof} The proof is simply a question of pulling together the linear algebra. We split the eact sequence of homology of the graph to write $\R^E = H_1(G,\R)\oplus \R^{V,0}$. With respect to this basis, the symmetric matrix corresponding to the quadratic form $\sum y_ee^{\vee,2}$ is
\eq{}{\begin{pmatrix}M_y & B_y^{\delta,t} \\ B_y^{\mu} & p^{\mu}\Gamma_yp^{\delta,t} \end{pmatrix}
}

Here the pairing on the $b_i$ is $M_y=\sum y_eM_e$ with $M_e$ as in remark \ref{rmk25}. One external momentum maps to $\sum_e \langle\sigma_\delta,a^\delta_e\rangle e$. Pairing this with the $b_j$ yields  the column vector $B_y^{\delta,t} = (\dotsc,\sum_e y_e\langle\sigma_\delta,a^\delta_e\rangle c_{e,j},\ldots)^t$.  The other external momentum pairs with the $b_j$ to yield a row vector $B_y^{\mu} = (\dotsc,\sum_e y_e\langle\sigma_\mu,a^\delta_e\rangle c_{e,j},\ldots)$. Finally, the entry in the bottom right corner is $\sum_e y_e\langle \sigma_{\delta},a^\delta_e\rangle  \langle \sigma_{\mu},a^\mu_e\rangle$. The theorem follows from remark \ref{det_formula}. 
\end{proof}

To summarize, our recipe for computing the Feynman amplitude \eqref{35a} is the following. Given the external momentum $p:=\sum_{v\in V(G)} r_v v$ with $r_v \in \R^D$ and $\sum_v r_v=0$, choose an analytic neighborhood $S$ of $C_0$ in a suitable local moduli space of marked curves and two relative $\R^D$-divisors $\frak A_i:= \sum_v r_v\sigma_{v,i},\ i=1,2$ supported on the markings. We assume the markings $\sigma_{v,i}$ are all disjoint. Further, we require that the intersection numbers with components of the rational curve $C_0$ coincide, and that $\sigma_{v,1}\cdot C_{0,v'}=\sigma_{v,2}\cdot C_{0,v'}= \delta_{v,v'}$. (Kronecker delta.) (The issue here is that the height $\langle \frak A_1, \frak A_2\rangle$ is only defined if the supports of $\frak A_1,\ \frak A_2$ are disjoint.) Recall from the beginning of this section we have 
\ga{}{\widetilde U=\H^{E(G)}\times \Delta^{3g-3-E(G)+N} \to U^*=\Delta^{*,E(G)}\times \Delta^{3g-3-E(G)+N}; \notag \\
X_e+iY_e/\alpha' = z_e=\frac{\log(t_e)}{2\pi i}.\notag
}
Then (compare \eqref{7.8}. The factor $2$ occurs because of \eqref{7.1}.)
\eq{}{\lim_{\alpha' \to 0} \alpha' \mathbf{S}[C_{t_e},\frak A_1,\frak A_2] = 2\phi_G(p,\{Y_e\})/\psi_G(\{Y_e\}).
}
The Feynman amplitude arises by plugging this limit into the exponential in \eqref{35a} and integrating over the the space of nilpotent orbits which is parametrized by $Y_e\ge 0$. 

Let me suggest two directions for further research.\nn
1. The rational curve $C_0$ with dual graph $G$ corresponds to a maximally degenerate boundary point on the moduli space of (marked) curves of genus $g$. Are there physically meaningful quantities arising when we specialize the height action \eqref{7.8} to less degenerate boundary points. The classical interpretation of Feynman amplitudes as coefficients of Green's functions arising from path integrals is really more of a metaphor than a mathematically rigorous calculus. Perhaps studying partial degenerations might throw more light. \nn
2. The height $\langle\frak A_1, \frak A_2\rangle$ is just one example of an $\R$-valued function associated to a variation of mixed Hodge structures. (In this case, we consider variations of biextensions.) Are there other interesting quantities arising from integration over the space of nilpotent orbits when a variation of Hodge structures degenerates?

\section{Cutkosky Rules} \label{cutrules}
The final topic in the notes is a report on some aspects of joint work with Dirk Kreimer on Cutkosky rules. The amplitude $A_G= A_G(q)$ is a multi-valued function of external momenta $q$, as one sees e.g. from \eqref{prodquads}. (For simplicity the masses are kept fixed.)  The following beautiful physical interpretation of the {\it variation} $\text{Var}(A_G)$ as $q$ winds around the threshold locus was given in \cite{Cut}, \cite{ELOP}:
\eq{10.1}{\text{Var}(A_G) = (-2\pi i)^r \int\frac{\delta^+(q_1^2-m_1^2)\cdots \delta^+(q_r^2-m_r^2)d^{Dg}x}{(q_{r+1}^2-m_{r+1}^2)\cdots (q_n^2-m_n^2)}.
} 
Here $A_G$ is the amplitude associated to a graph $G$ with $n$ edges and $g$ loops. The $q_i^2-m_i^2$ are propagators, and $\delta^+(q_i^2-m_i^2)$ means to take the residue of the original form $\omega:= \frac{d^{Dg}x}{\prod_1^n(q_j^2-m_j^2)}$ along the divisor $q_i^2-m_i^2=0$ and then to integrate over the part of the real locus of that intersection of divisors where the energy $q_i^{(0)}>0$. External momenta are placed near the threshold divisor associated to a pinch point of the intersection $\bigcap_1^r\{q_i^2-m_i^2=0\}$.

The basic monodromy calculation was related to the classical Picard-Lefschetz formula in mathematics by Pham, \cite{Ph}. Pham gives a complete description of the local topology of the union of propagator quadrics as the external momentum approaches a general threshold point. However, the physical formula \eqref{10.1} involves the real structure of the quadrics, and more work is necessary to link the topology to the Cutkosky formula.  In the following I sketch the arguments involved in proving Cutkosky's formula for the case of {\it physical singularities} (definition \ref{physsing}). 

The quadrics have the form $q^2-m^2=(e^\vee-c_e)^2-m^2$ where $e\in E=E(G)$ is an edge, $c_e\in \C^D$ is a $D$-tuple of constants, and $e^\vee$ is a $D$-tuple of functions on $\C^{Dg}$:
\eq{10.2}{\C^{Dg}=H_1(G,\C^D) \inj (\C^D)^E \xrightarrow{e^\vee-c_e} \C^D \xrightarrow{()^2-m^2} \C.
}
Write $Q_e$ for the quadric in $\C^{Dg}$ defined by \eqref{10.2}. Given edges $e_1,\dotsc,e_r$, it is straightforward to check that the $rD$ linear functions $(e_1^\vee,\dotsc,e_r^\vee): H_1(G,\C^D) \to \C^{rD}$ form part of a system of coordinates on $\C^{Dg}$ if and only if this map is surjective, and that this condition is independent of $D$. Taking $D=1$, one sees that the linear functions do not form part of a system of coordinates if and only if the graph obtained from $G$ by cutting $e_1,\dotsc,e_r$ is not connected. On the other hand, if the linear functions form part of a system of coordinates and if none of the masses $m_i=0$, it is clear that the intersection of the propagator quadrics cannot have a pinch point (irrespective of the $c_{e_i}$). We conclude
\begin{prop} With notation as above, assuming all masses $m_i\neq 0$, a necessary condition for the intersection of the quadrics 
$$Q_{e_1}\cap\cdots\cap Q_{e_r}$$ 
to have a pinch point for suitable values of the $c_{e_i}$ is that the graph obtained from $G$ by cutting $e_1,\dotsc,e_r$ be disconnected. 
\end{prop}

For simplicity, in what follows I consider only the case $r=n$, i.e. the case when all the edges are placed on shell. The resulting cut graph is just the collection $V(G)$ of vertices and is certainly disconnected (assuming $G$ has at least $2$ vertices). 
\begin{ex}[Banana graphs] Suppose $G$ has $2$ vertices connected by $n$ edges $e_1,\dotsc,e_n$. The only possibility to disconnect $G$ is to cut all the edges. Since cutting $e_1,\dotsc,e_{n-1}$ leaves $G$ connected, it follows as above that $e_i^\vee: \C^{D(n-1)} \to \C^D,\ 1\le i\le n-1$ are independent coordinates, and upto orientation $e_n^\vee = \sum_1^{n-1} e_i^\vee - a$ for some $a\in \C^D$. We can replace the $e_i^\vee$ by $e_i^\vee-c_i$ and arrange that the jacobian matrix for the $n$ quadrics looks like
\eq{}{2\begin{pmatrix}e_1^\vee & 0 & 0 & \hdots \\0 & e_2^\vee & 0 & \hdots \\
\vdots & \vdots & \vdots & \vdots \\
0 & \hdots & 0 & e_{n-1}^\vee \\
\sum_1^{n-1}e_i^\vee + a & \sum_1^{n-1}e_i^\vee + a & \sum_1^{n-1}e_i^\vee + a &\sum_1^{n-1}e_i^\vee + a 
\end{pmatrix}
}
(Note each entry is a $D$-vector.) We want this matrix to have rank $n-1$ at some point where the $n$ propagator quadrics vanish. A small exercise in linear algebra (for details, see \cite{BK}, section $12$) yields 
\eq{}{|a| = \sum_1^n \mu(i)m_i;\quad \mu(i) = \pm 1.
}
(Here $|a|:= \sqrt{a^2}$.) 
\end{ex}

\section{Cutkosky rules: Pham's vanishing cycles}\label{pham}

In this section I summarize briefly the calculus of vanishing cycles; first the classical theory and then the theory of F. Pham. Classically, $f: \sX \to \Delta$ is a proper family of varieties over a disk $\Delta\subset \C$. We assume $\sX$ is non-singular, and $f$ is smooth away from $0\in \Delta$. Further, $\sX_0=f^{-1}(0)$ has a single ordinary double point $s_0$ and is non-singular elsewhere. With these hypotheses, we can find analytic coordinates $x_1,\dotsc,x_n$ near $x_0$ and a coordinate $t$ on $\Delta$ such that locally the family is defined by
\eq{}{x_1^2+\ldots + x_n^2=t.
}
For $0<\ve<<1$, the homology of a ball $B$ around $s_0$ intersected with the fibre $\sX_\ve$ is generated by the class of the real sphere $S_\ve: \sum x_i^2=\ve, \ x_i\in \R$. The monodromy action on the smooth fibre $H^*(\sX_\ve,\Q)$ is given by the Picard-Lefschetz formula
\eq{}{c \mapsto c+(-1)^{n(n+1)/2}\langle ex(c),S_\ve\rangle i_*S_\ve.
}
Here
\eq{}{ex: H_*(\sX_\ve) \to H_*(B\cap \sX_\ve,\partial(B\cap \sX_\ve))
}
denotes excision, and $i_*: H_*(B\cap \sX_\ve) \to H_*(\sX_\ve)$ is the pushforward. The pairing $\langle ex(c),S_\ve\rangle$ refers to the natural pairing
\eq{}{H_{n-1}(B\cap \sX_\ve,\partial(B\cap \sX_\ve)) \otimes H_{n-1}(B\cap \sX_\ve) \to \Q. 
}

To explain Pham's vanishing cycles, we change notation. Let $f:M \to P$ be a smooth, proper map of algebraic varieties. Let $S = \bigcup_{i=1}^n S_i\subset M$ be a normal crossings divisor. Let $s_0\in \bigcap_{i=1}^nS_i \subset S\subset M$, and let $p_0=f(s_0)$. We will work locally near $s_0$ on $M$, so we may fix analytic coordinates $t_1,\dotsc,t_k$ around $p_0$ on $P$ and $x_1,\dotsc,x_\ell$ around $s_0$ on $M$ in such a way that $t_1,\dotsc,t_k,x_1,\dotsc,x_\ell$ form a full set of coordinates on $M$ near $s_0$. We assume that locally 
\ga{}{S_i: x_i=0;\quad i=1,\dotsc,n-1, \\
S_n: t_1-(x_1+\cdots+x_{n-1}+x_n^2+\cdots+x_\ell^2)=0. 
}
For such a configuration of divisors, the point $s_0$ (origin) is called a {\it pinch point}. Notice that on the smallest stratum $\bigcap_1^n S_i$ of $S$, viewed as a variety fibred over $P$, we have a classical Picard-Lefschetz vanishing cycle local equation, which is a family of spheres $x_n^2+\cdots+x_\ell^2 = t_1$ degenerating as $t\to 0$. For $t_1=\ve>0$, the {\it vanishing sphere} $\rm{vsphere}_\ve$ is the real sphere 
\eq{}{\rm{vsphere}_\ve: \ve=x_n^2+\cdots+x_\ell^2;\quad x_i\in \R,\ n\le i\le \ell. 
}
Classically, $\rm{vsphere}_\ve$ is referred to as the vanishing cycle, but we follow Pham here and distinguish $3$ different topological chains, the vanishing sphere, the vanishing cell, and the vanishing cycle. 
\begin{defn}With notation as above, the vanishing cell, 
\eq{}{\rm{vcell}_\ve:x_i \ge 0,\ 1\le i\le n-1;\quad \ve-( x_n^2+\cdots+x_\ell^2) \ge 0.
}
The vanishing cycle $\rm{vcycle}_\ve$ is the iterated tube 
$$\tau_{1,\ve_1}^*\tau_{2,\ve_2}^*\cdots\tau_{n-1,\ve_{n-1}}^*(\rm{vsphere}_\ve).
$$
Here $\ve>>\ve_{n-1}>>\cdots >>\ve_1>0$. The notation $\tau_{i,\ve_i}^*$ refers to pulling back to the circle bundle of radius $\ve_i$ inside the (metrized) normal bundle for $S_1\cap\cdots\cap S_i \subset S_1\cap \cdots \cap S_{i-1}$. This circle bundle is viewed as embedded in $S_1\cap\cdots\cap S_{i-1}$ and not meeting $S_1\cap\cdots\cap S_{i}$. The inequalities $\ve_i<<\ve_{i+1}$ insure that $\rm{vcycle}_\ve$ is a closed chain on $M-\bigcup_{i=1}^n S_i$. 
\end{defn}

Here are two pictures in the case $n=2$. 

\includegraphics[scale=.5]{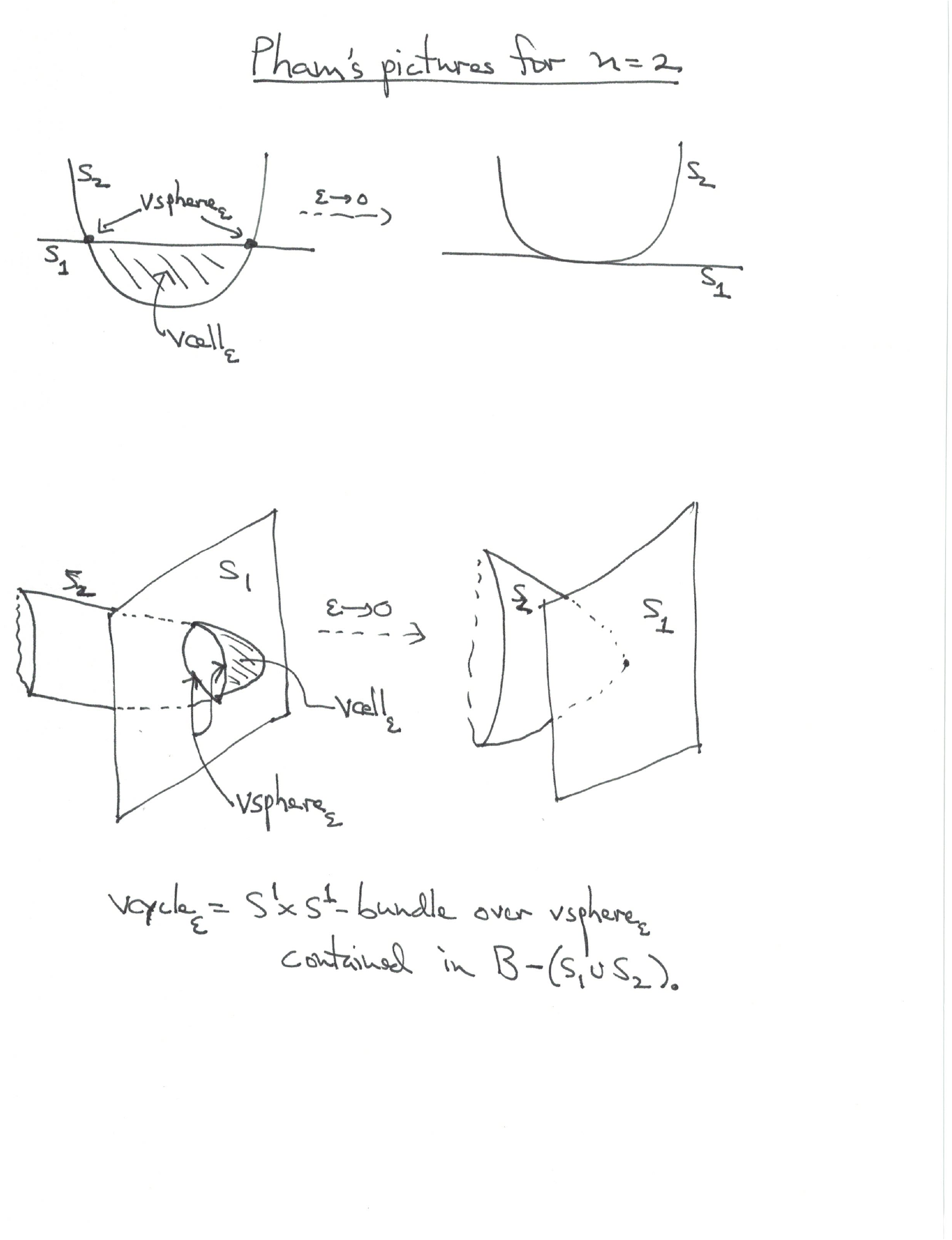}

The local structure of homology around $s_0$ is computed by Pham to be
\begin{thm}\label{homology}Let $B$ be a small ball around $s_0$ in $M$. Let $p\in P$ be near $p_0$ and assume $t_1(p)=\ve>0$. Write $d=\dim_\C M_p$.
\nn
(i) We have for reduced homology 
$$\widetilde H_j\Big(\bigcap_1^nS_{i,p}\cap B,\Z\Big)=(0),\ j\neq d-n:= \dim_\C \Big(\bigcap_1^nS_{i,p}\Big).
$$
$$\widetilde{H}_{d-n}\Big(\bigcap_1^nS_{i,p}\cap B,\Z\Big)=\Z\cdot \rm{vsphere}_\ve.
$$
\noindent (ii) In relative homology
$$H_j(B_p,S_p\cap B) = (0), \ j\neq d;\quad H_{d}(B_p,S_p\cap B) =\Z\cdot \rm{vcell}_\ve.
$$
\noindent (iii) For the homology of the complement
$$H_j(B_p-S_p\cap B) = (0),\ j\neq d;\quad H_{d}(B_p-S_p\cap B) =\Z\cdot \rm{vcycle}_\ve.
$$
\end{thm}

Local Poincar\'e duality yields a pairing
\eq{}{\langle\cdot,\cdot\rangle: H_d(B_p-S_p\cap B)\otimes H_{d}(B_p,(S_p\cap B)\cup \partial B_p) \to H_{2d}(B_p,\partial B_p) \cong \Z.
}
Let $\gamma$ be a simple closed path on $P$ based at $p$, supported in a neighborhood of $p_0$, and looping once around the divisor $t_1=0$. Using Thom's isotopy lemma \cite{M} and assuming that $f$ is a submersion on strata except at the point $s_0$, Pham shows that the variation of monodromy
\eq{}{var:= \gamma_*-Id:H_*(M_p-S_p)\to H_*(M_p-S_p)
}
factors through excision and a local variation map $var_{loc}$
\ml{}{ H_d(M_p-S_p) \xrightarrow{excision} H_d(B_p-S_p\cap B_p,\partial B_p) \xrightarrow{var_{loc}}\\
 H_d(B_p-S_p) \xrightarrow{i_*} H_d(M_p-S_p).
}
(The variation map is zero in homological degrees $\neq d$.) 

The Picard-Lefschetz theorem in this setup is
\begin{thm}\label{pl_thm} We have
\ml{}{var_{loc}(\rm{excision}(c)) = \rm{excision}(c) \\
+(-1)^{(n+1)(n+2)/2}\langle \rm{excision}(c),\rm{vcell}_\ve\rangle \rm{vcycle}_\ve
}
\end{thm}
\begin{proof} See \cite{Ph}. 
\end{proof}

\begin{ex}In this example we compute the Picard-Lefschetz transformation arising in physics. One has a set of quadrics $Q_i:f_i=0,\ 1\le i\le n$ indexed by the edges of a graph. Our normal crossings divisor in this case is $\bigcup Q_i$. We consider the situation locally around a pinch point $s_0$, and we write in Pham's coordinates, $Q_i: x_i=0,\ 1\le i\le n-1$ and $Q_n: t_1-(x_1+\cdots +x_{n-1} + \sum_n^\ell x_i^2)=0$. Notice that in these coordinates, the bad fibre over $p_0$ (where $t_1=0$) of $\bigcap_1^n Q_{i,p_0}$ has an isolated $\R$-point at the origin. We will see (proposition \ref{hessianprop}) in the case of {\it physical singularities} that this is also the case for the space-time coordinates. Hence it is plausible to assume that at least for physical singularities Pham's local coordinates are defined over $\R$. We are interested in computing the monodromy on $M-\bigcup Q_i$ of the chain given by taking space-time coordinates in $\R$. As it stands this makes no sense, because our quadrics are defined using the Minkowski metric so the $Q_i$ meet the real locus. The standard ploy to avoid this problem is to replace the defining equations $f_j=0$ by $f_j+ia_j=0$ with $1>>a_j>0$.  Using theorem \ref{pl_thm}, we need to compute 
$$\langle \rm{excision}(\R^{Dg}),\rm{vcell_\ve}\rangle,$$ 
where $D$ is the dimension of space-time and $g$ is the number of loops of the graph. 

In Pham's local coordinates, the deformed quadrics are defined by 
$$x_j=-ia_j,\ 1\le j\le n-1\text{ and }\ve + ia_n - (\sum_1^{n-1}x_j+\sum_n^{Dg} x_j^2)=0.$$ 
We let $x_j$ run over the path 
$$x_j=(1-\rho_j)ia_j + \rho_j(\ve+ia_n)/(n-1),\ 1\le j\le n-1.$$ 
Here $0\le \rho_j\le 1$. For points $x^0=(x_1^0,\dotsc,x^0_{n-1})$ on those paths, the quadratic equation becomes 
$$\sum_n^{Dg} x_j^2 = \ve + ia_n - \sum_1^{n-1}x_j^0=:r(x^0)e^{i\theta(x^0)}.$$
The locus
\eq{17}{\{(e^{i\theta(x^0)/2}u_n,\dotsc e^{i\theta(x^0)/2}u_{Dg})\ |\ \sum_{k=n}^{Dg}u^2_k \le r(x^0)\}
} 
is a solid sphere, and the union of those solid spheres as the $x_j$ run over those paths is the chain $\text{vcell}_\ve$ for the deformed quadrics. 

According to theorem \ref{pl_thm}, the multiplicity of the vanishing cycle in the variation of the cycle $\R^{Dg}$ (or more precisely of $\P^{Dg}(\R)$) around the given pinch point is the intersection of this real locus with $\text{vcell}_\ve$.

Note that $x_j$ is real only at the point $\rho_j=a_j/(a_j+a_n/(n-1))\in (0,1)$. At the point with these coordinates, the quadratic equation reads $\sum_n^{Dg} x_j^2 = R+ia_n$ for some $R\in \R$. since $a_n\neq 0$, the only real point on the corresponding solid sphere \eqref{17} lies at the origin. Thus, $\text{vcell}_\ve\cap \R^{Dg}$ is precisely one point. The intersection looks locally like the intersection $\R^N\cap e^{i\theta}\R^N \subset \C^N$, so it is transverse, and $$\langle\rm{excision}(\P^{Dg}(\R)),\rm{vcycle}\rangle = \pm 1.
$$ 
We do not try to compute the sign. 

This example is important because it proves that there really is non-trivial monodromy associated to the physical situation. Note the computation is very dependent on taking the $a_j>0$. 
\end{ex}
\section{Cutkosky's Theorem} \label{cutthm}
We want to prove \eqref{10.1} in the case of {\it physical singularities}. (Recall, we are also assuming for simplicity that we cut all edges so $r=n$ in \eqref{10.1}.) To understand physical singularities, consider formula \eqref{5.20} for the amplitude $A_G$. Viewed as a function of external momenta, $A_G$ will be non-singular unless the chain of integration $\R^{Dg}\times \sigma$ meets the polar locus $\sum_{e\in E} A_e=0$. When it does meet, it may still happen that the chain can be deformed away from the polar locus so $A_G$ has no singularity. 
The logical place to look for branchpoints of $A_G$ is thus at external momenta where $\sigma\times \R^{Dg}$ contains singular points of the universal quadric
\eq{}{Q: \sum_{e\in E} A_ef_e=0.
}
\begin{defn}\label{physsing} A point $p= (c_1,\dotsc,c_n,p') \in Q(\R)\subset \R^n\times\R^{Dg}$ is a physical singularity if $p$ is a singular point of $Q$ and all $c_i>0$. By extension, a point $p'\in \R^{Dg}$ is a physical singularity if there exists a physical singularity $(c,p')\in Q$. 
\end{defn}
\begin{lem}Let $p=(c_1,\dotsc,c_n,p')$ be a physical singularity of the universal quadric $Q$. Then $p'\in \bigcap_1^n \{f_i=0\}$ and $p'$ is a singular point in that intersection of quadrics. 
\end{lem}
\begin{proof}Let $f=\sum A_if_i$. By assumption, $p$ is singular so $f_i(p')=\frac{\partial}{\partial A_i}f(p) = 0$. Also, vanishing of the partials of $f$ with repsect to coordinates on $\R^{Dg}$ yields the relation
\eq{gradrel}{\sum_1^n c_i \text{grad}_{p'}(f_i)=0,
}
which implies that $p'$ is singular on the intersection of quadrics. 
\end{proof}
\begin{defn}A physical singularity $p$ will be called non-degenerate if two conditions hold:\nn
(i) The relation \eqref{gradrel} is the only non-trivial relation among the gradients at $p'$. \nn
(ii) The Hessian matrix at the singular point for the intersection of the quadrics is non-degenerate. 
\end{defn}

Condition (i) in the above definition means we can choose a partial set of local coordinates $x_1,\dotsc,x_{n-1}$ at the singular point $p'$ in such a way that the intersection of the quadrics is cut out locally near $p'$ by $x_1=\cdots = x_{n-1}=g=0$ for some function $g$. Condition (ii) says that the matrix of second order partials of $g|\{x_1=\cdots = x_{n-1}=0\}$ is non-degenerate at $p'$.  

In order to understand the $\R$-structure we borrow a standard picture for a Morse function $f$, \cite{Mil}.
\begin{figure}[h]
\includegraphics[scale=.5,clip,trim=0mm 80mm 0mm 0mm]{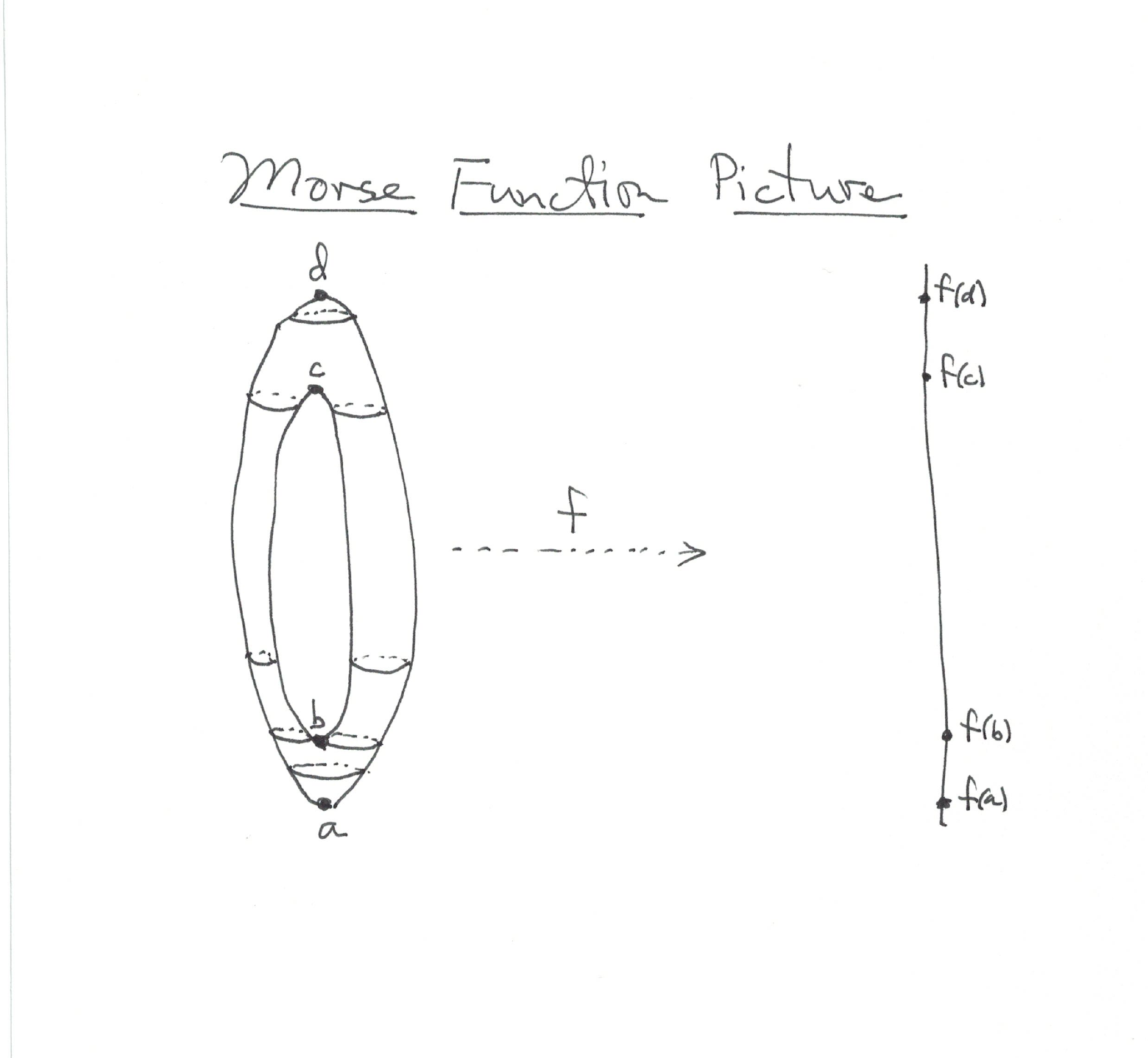}
\centering
\end{figure}

The Hessian matrix of $f$ at the points $a, d$ is definite, while the Hessian matrix at $b,c$ is indefinite. Note $f^{-1}f(a)=a$ and the fibre $f^{-1}(x)$ for $x$ slightly above $f(a)$ is a sphere (a circle in the picture). This is the vanishing cycle which contracts to the point $a$ as $x\to f(a)$. The picture at $d$ is similar for $y$ slightly below $f(d)$. On the other hand, $f^{-1}f(b)$ and $f^{-1}f(c)$ are figure-eights, and there are no vanishing cycles.  In our case, if we know the Hessian for the function $g$ at $p'$ is either positive or negative definite, then Pham's vanishing cycle will coincide with the real sphere given by Morse theory. 

\begin{prop}\label{hessianprop} Let $p=(c_1,\dotsc,c_n,p')$ be a non-degenerate physical singularity. Assume $m_i^2>0$ for all masses. Then the Hessian associated to $p'\in \bigcap_1^n \{f_i=0\}$ is negative definite. In particular, Pham's vanishing sphere is real, and Cutkosky's theorem \eqref{10.1} holds for the variation of $A_G$.   
\end{prop}
\begin{proof}Order the edges $e_1,\dotsc,e_n$ such that $e_1^\vee,\dotsc,e_g^\vee$ give coordinates on $H_1(G,\R^D)$. (Note each $e_i^\vee=(e_i^{\vee,1},\dotsc, e_i^{\vee,D})$ is a $D$-tuple of coordinate functions.) We change notation slightly and write 
$$p=(c_1,\dotsc,c_n,e_1^\vee(p'),\dotsc,e_g^\vee(p'))\in \R^n\times \R^{Dg}.
$$
The quadrics have the form $f_i = (e_i^\vee-s_i)^2-m_i^2$ for some $s_i\in \R^D$.  By a Euclidean change of variables, we can take $s_i=0$ for $1\le i\le g$. Because $p$ is singular on the universal quadric $Q$, the sum $\sum c_if_i$ is pure quadratic when expanded about $p'$, i.e. constant and linear terms vanish. Notice that the quadratic part of $f_i$ is $(e_i^\vee)^2$ which is a Minkowski square. The quadratic form associated to the Hessian is then simply $\sum c_i(e_i^\vee)^2$. Of course, these are Minkowski squares, so it is not possible to say anything about the sign of the form on $\R^{Dg}$. However, what we want is the sign of the form restricted to the intersection of the tangent spaces at $p'$ of the $f_i$. Among these are the tangent spaces at $p'$ to $f_i=(e_i^\vee)^2-m_i^2$ for $1\le i\le g$. The tangent space at $p'$ for such a $f_i$ is $(e_i^\vee)^{-1}(e_i^\vee(p')^\perp)\subset \R^{Dg}$. (Here $e_i^\vee(p')^\perp \subset \R^D$. Note $(e_i^\vee(p'))^2=m_i^2>0,\ 1\le i\le g$.) Since the Minkowski metric has sign $(1,-1,\dotsc,-1)$ on $\R^D$, we conclude that the sign of the metric on the intersection of these spaces is negative definite. In particular, $\sum_1^n c_i(e_i^\vee)^2$ is certainly negative semi-definite on the intersection of the tangent spaces. (Indeed, we can write $e_j^\vee = \sum_{i=1}^g \mu_j^i e_i^\vee$ with $\mu_j^i\in \R$. This decomposition is orthogonal, so $(e_j^\vee)^2 = \sum_{i=1}^g (\mu_j^i)^2 (e_i^\vee)^2$. Since the $(e_i^\vee)^2$ on the right are negative definite on the intersection of the tangent spaces, $(e_j^\vee)^2$ is negative definite as well.) The only way $\sum_1^n c_i(e_i^\vee)^2$ could fail to be negative definite would be if there was a non-zero tangent vector $t$ such that $e_i^\vee(t)=0,\ 1\le i\le n$. But this is not possible since the $e_i^\vee,\ 1\le i\le g$, give coordinates. 
\end{proof}

\bibliographystyle{amsplain}


\end{document}